\documentclass[smallextended]{svjour3}       
\usepackage{amsfonts}
\usepackage{amsmath}
\usepackage{amssymb}
\usepackage{array}
\usepackage{stmaryrd}
\usepackage{graphicx}

\usepackage{mathabx} 
\usepackage{bbm}

\usepackage{graphicx}
\usepackage{tikz}
    \usepackage{color}  
     \definecolor{red}{rgb}{0.9,0,0}
     \definecolor{green}{rgb}{0,0.6,0}
     \definecolor{rb}{rgb}{0.6,0,0.2}     
     \definecolor{blue}{rgb}{0,0,0}

\renewcommand{\sim}{\simeq}
\renewcommand{\epsilon}{\varepsilon}

\newcommand{\pt}{\partial}
\newcommand {\eps} {\varepsilon}

\newcommand{\R}{\mathbb{R}}

\newcommand{\PP}{{\mathcal P}}
\newcommand{\EE}{{\mathcal E}}
\newcommand{\hatE}{{\widetilde{\mathcal E}}}

\newcommand{\bnu}{{\boldsymbol{\nu}}}

\newcommand{\proofend}{\hfill$\Box$}

\newcommand {\beq} {\begin{equation}}
\newcommand {\eeq} {\end{equation}}
\newcommand {\beqa} {\begin{eqnarray}}
\newcommand {\eeqa} {\end{eqnarray}}
\newcommand {\beqann} {\begin{eqnarray*}}
\newcommand {\eeqann} {\end{eqnarray*}}

\numberwithin{equation}{section}
\numberwithin{remark}{section}
\numberwithin{lemma}{section}
\newtheorem{theorem_}[lemma]{Theorem}
\newtheorem{corollary_}[lemma]{Corollary}

\begin{document}

\title{Lower a posteriori error estimates
on
anisotropic meshes\thanks{The author was partially supported by  Science Foundation Ireland grant SFI/12/IA/1683.}
}

\titlerunning{Lower a posteriori error estimates on anisotropic meshes}        

\author{Natalia Kopteva}


\institute{N. Kopteva \at Department of Mathematics and Statistics,
University of Limerick,
Limerick, Ireland
\\
\email{ natalia.kopteva@ul.ie}}

\date{}

\maketitle

\begin{abstract}
Lower a posteriori error bounds obtained using the standard bubble function approach are reviewed
in the context of anisotropic meshes.
A~numerical example is given  that clearly demonstrates that the short-edge jump residual terms in such bounds are not sharp.
Hence,
for linear finite element approximations of the Laplace equation in polygonal domains,
a new approach is employed to obtain essentially sharper lower a posteriori error bounds and thus to show that the upper error estimator in the recent paper \cite{Kopt_NM_17} is efficient on
{\color{blue}partially structured}
anisotropic meshes.
\keywords{%
Anisotropic triangulation \and Lower a posteriori error estimate \and
        Estimator efficiency}
 \subclass{65N15 \and 65N30}
\end{abstract}


\section{Introduction}

The purpose of this paper is to address the efficiency of a posteriori error estimators on anisotropic meshes,
which essentially reduces
to obtaining sharp lower a posteriori error bounds.
For shape-regular meshes such lower error bounds can be found in \cite{AinsOd_2000,Ver_book_13}.
For anisotropic meshes, the situation is more delicate, as we shall now elaborate.

For unstructured anisotropic meshes,
both upper and lower a posteriori error estimates were obtained
 in \cite{Kunert2000,Kun01,KunVer00} 
for the Laplace equation and 
for a singularly perturbed reaction-diffusion equation; see also \cite[\S4.5]{Ver_book_13}.
{\color{blue}We also refer the reader to
\cite{Mich_Perrotto},
where the reliability and efficiency
of a residual-type estimator from \cite{Picasso_2003} based on the
Zienkiewicz-Zhu recovery procedure
was established on anisotropic meshes
under an $\eta$\% superconvergence type condition (explained, e.g., in \cite[\S4.8]{AinsOd_2000}).}
It should be noted that although the lower  error bounds in \cite{Kunert2000,Kun01,KunVer00} involve the same estimators as the corresponding upper bounds,
however
 the error constants in the upper bounds include the so-called matching functions. The latter depend on the unknown error
and take moderate values only when the mesh is either isotropic, or, being anisotropic, is aligned correctly to the solution,
while, in general, they may be as large as mesh aspect ratios.

The presence of such matching functions in the estimator is clearly undesirable. It is entirely avoided in the more recent papers \cite{Kopt15,Kopt_NM_17,Kopt17},
where upper a posteriori error estimates
on anisotropic meshes
were  obtained for
singularly perturbed semilinear reaction-diffusion equations in the energy norm and in the maximum norm.

Interestingly, the efficiency of the estimators in \cite{Kopt15,Kopt_NM_17,Kopt17} cannot be established using the standard bubble function approach, employed in \cite{Kunert2000,Kun01,KunVer00}.
To be more precise, this approach (which will be reviewed in \S\ref{sec_kunert})
leads to lower error bounds
with significantly smaller weights at the short-edge jump residual terms than those in the upper bounds.

The main findings of the {\color{blue}present} paper are as follows.\vspace{-0.2cm}

\begin{itemize}
\item
Lower a posteriori error bounds obtained using the standard bubble function approach, such as in \cite{Kunert2000,Kun01,KunVer00}, will be reviewed
in the context of anisotropic meshes.
Numerical examples will be given in \S\ref{sec_kunert} that clearly demonstrate that the short-edge jump residual terms in such bounds are not sharp.
\smallskip
\item
Hence, we shall present
a new approach  that yields essentially sharper lower a posteriori error bounds and thus shows that the upper error estimator in \cite{Kopt_NM_17} is efficient on
{\color{blue}partially structured} 
anisotropic meshes.

\end{itemize}\vspace{-0.1cm}

{\color{blue}Note that mild restrictions on the structure of the mesh are  not uncommon in the literature when, for example, recovery type a posterior error estimators are considered  \cite{xu_Zhang}, and, as discussed in \S\ref{ssec_gen_mesh_main_res}, such restrictions are not unreasonable when an anisotropic mesh is generated starting from a regular mesh.}

Compared to \cite{Kopt15,Kopt_NM_17,Kopt17},  to simplify the presentation,
we shall restrict the consideration to the simpler Laplace equation and
consider 
the  problem
\begin{align}
-\triangle u = f(x,y) \quad\mbox{for}\;\;(x,y)\in \Omega, \qquad u=0\quad \mbox{on}\;\; \partial \Omega,
\label{eq1-1}
\end{align}
posed in a,
possibly non-Lipschitz,
polygonal domain $\Omega\subset\mathbb{R}^2$.
 We also assume that
 $f \in L_\infty(\Omega)$ (for a less smooth $f$, see  Remarks~\ref{rem_f_I} and~\ref{rem_f_I_post}).

Linear finite element approximations  of~\eqref{eq1-1} will be considered.  Let $S_h \subset H_0^1(\Omega)\cap C(\bar\Omega)$ be
a piecewise-linear finite element space  relative to a triangulation $\mathcal T$, and let
the computed solution $u_h \in S_h$ satisfy
\begin{align}
\langle\nabla u_h,\nabla v_h\rangle=\langle f,v_h\rangle\qquad\quad
\forall\;v_h \in S_h,
\label{eq1-2}
\end{align}
where $\langle\cdot,\cdot\rangle$ denotes the $L_2(\Omega)$ inner product.

To give an idea of the results in \cite{Kopt_NM_17},
under the assumptions on the mesh described in \S\ref{sec_nodes},
 one upper error estimate reduces to \cite[Theorems~6.1 and~7.4]{Kopt_NM_17}
\beq
\label{first result}
\| \nabla(u_h-u) \|_{2\,;\Omega}\le C\,
\Bigl\{\sum_{S\in{\mathcal S}\backslash\pt\Omega}\!\!\!|\omega_S|J_S^2
+
\sum_{T\in{\mathcal T}}
\bigl\|H_Tf^I\bigr\|^2_{2\,;T}
+\,\bigl\|f-f^I\bigr\|^2_{2\,;\Omega}\Bigr\}^{1/2}\!\!,
\eeq
where $C$ is 
{independent   of the diameters and the aspect ratios} of elements in $\mathcal T$.
Here
$\mathcal S$ is the set of edges 
in $\mathcal T$,
$J_S$
is the standard jump in the normal derivative of $u_h$ across  any interior edge $S\in{\mathcal S}\backslash\pt\Omega$, and
$\omega_S$ is the patch of two elements sharing $S$.
We also use $H_T:={\rm diam}(T)$, which may be significantly larger than
$h_T:= 2H_T^{-1}|T|$,
and  the standard piecewise-linear Lagrange interpolant $f^I\in S_h$ of~$f$.

Furthermore, under some additional assumptions on the orientation of mesh elements surrounding sequences of anisotropic nodes 
connected by short edges, a sharper upper estimator was obtained in \cite[Theorem~6.2]{Kopt_NM_17}:
\begin{align}
\notag
\| \nabla(u_h-u) \|_{2\,;\Omega}\le C
\,
\Bigl\{\sum_{S\in{\mathcal S}\backslash\pt\Omega}\!\!\!|\omega_S|J_S^2&{}+
\sum_{T\in{\mathcal T}}\bigl\|h_Tf^I\bigr\|^2_{2\,;T}
+\bigl\|f-f^I\bigr\|^2_{2\,;\Omega}
\\[-0.1cm]\label{first result_improved}
&{}+\sum_{T\in{\mathcal T}}\bigl\|H_T{\rm osc}(f^I\,;T)\bigr\|^2_{2\,;T}\,\Bigr\}^{1/2}\!\!.
\end{align}

To relate \eqref{first result} and \eqref{first result_improved} to interpolation error bounds, as well as to possible adaptive-mesh construction strategies, note that
$|J_S|$ may be interpreted as approximating the diameter of $\omega_S$ under the  metric
induced by the squared Hessian matrix of the exact solution
(while $f^I$ approximates $\triangle u$).

Our task in this paper will be to establish the efficiency of the upper estimator in \eqref{first result_improved} up to data oscillation.
As was already mentioned, the standard bubble function approach yields unsatisfactory lower bounds, with the weight $\frac{|S|}{{\rm diam}(\omega_S)}|\omega_S|$ at $J_S^2$
(rather than a simpler and more natural $|\omega_S|$ in \eqref{first result_improved}).
Remark~\ref{rem_new_appr} sheds some light on our approach to
remedying this.


The paper is organized as follows.
In \S\ref{sec_kunert},
 we review lower a posteriori error bounds obtained using the standard bubble function approach.
In particular,  numerical examples are given  that  demonstrate that the short-edge jump residual terms in such bounds are not sharp.
The remainder of the paper is devoted to obtaining
sharper lower error bounds
In \S\ref{sec_nodes}, we describe basic triangulation assumptions.
Then
in
\S\ref{sec_struct}, 
we  present a version of the analysis for  partially structured meshes, while
the case of more general anisotropic meshes is addressed in \S\ref{sec_gen}.
\smallskip

{\it Notation.}
We write
 $a\sim b$ when $a \lesssim b$ and $a \gtrsim b$, and
$a \lesssim b$ when $a \le Cb$ with a generic constant $C$ depending on $\Omega$ and
$f$,
but 
%
not
 on 
 the diameters and the aspect ratios of elements in $\mathcal T$.
  Also, for $\mathcal{D}\subset\bar\Omega $ and $1 \le p \le \infty$,
  let $\|\cdot\|_{p\,;\mathcal{D}}=\|\cdot\|_{L_p(\mathcal{D})}$
   and $\|\cdot\|_{\mathcal{D}}=\|\cdot\|_{2\,;\mathcal{D}}$,
  and also ${\rm osc}(v\,;\mathcal{D})=\sup_{\mathcal{D}}v-\inf_{\mathcal{D}}v$ for $v\in L_\infty(\mathcal{D})$.
  Whenever quantities such as
  ${\rm osc}(\cdot\,;T)$ or $H_T$ appear in volume integrals or related norms,
  or $J_S$ appears in line integrals or related norms, 
they are understood as piecewise-constant functions.

\section{Standard lower error bounds are not sharp on anisotropic meshes}\label{sec_kunert}

This section is devoted to lower error bounds, such as in \cite{Kunert2000,Kun01,KunVer00}, obtained using the standard bubble function approach.
Numerical examples will be given in \S\ref{ssec_non_sharp} that clearly demonstrate that the short-edge jump residual terms in such bounds are not sharp.
These examples also suggest that the jump residual terms
 in our upper estimators \eqref{first result} and \eqref{first result_improved} have correct weights
 (the efficiency of the latter will be theoretically justified in \S\S\ref{sec_struct}-\ref{sec_gen}).
Furthermore, in \S\ref{ssec_kun_}, we shall review the bubble function approach when applied to anisotropic meshes and discuss its deficiencies
with a view of
changing the paradigm for deriving upper bounds for jump residuals associated with short edges (in particular, see Remarks~\ref{rem_def} and~\ref{rem_new_appr}).

\subsection{Numerical examples}\label{ssec_non_sharp}
Our first test problem is
(\ref{eq1-1}) with  the exact solution $u=\sin(\pi a x)$ (for $a=1,3$)  and
the corresponding $f$ in $\Omega=(0,1)^2$.
We employ the triangulation obtained by drawing diagonals
from the tensor product of the uniform grids $\{\frac{i}{N}\}_{i=0}^N$ and $\{\frac{j}{M}\}_{j=0}^M$  in the $x$- and $y$-directions respectively
(with all diagonals having the same orientation).
A standard quadrature
with $f$ replaced in \eqref{eq1-2} by its Lagrange interpolant $f^I\in S_h$
will be used in numerical experiments.

For this problem, we compare two lower error estimates:
obtained using the standard bubble function approach \cite{KunVer00} (see also Lemma~\ref{lem kun_lower} in \S\ref{ssec_kun_}) and  the one obtained in \S\ref{sec_struct}
(see Theorem~\ref{theo_lower_struct}). 
They can be described by
\begin{subequations}\label{two_lower_main}
\beq\label{two_lower}
{\mathcal E}:=
\Bigl\{\sum_{S\in{\mathcal S}\backslash\pt\Omega}\!\!\! \varrho_S\,|\omega_S| J_S^2+\|h_T f^I\|_{\Omega}^2\Bigr\}^{1/2}\!\!
\lesssim
\| \nabla(u_h-u) \|_{\Omega}+ \|h_T (f-f^I)\|_{\Omega},
\eeq
where  for the weight $\varrho_S$
for $S\in{\mathcal S}\backslash\pt\Omega$
we consider two choices:
\beq\label{rho_def}
\varrho_S=\left\{\begin{array}{cll}
\frac{|S|}{{\rm diam}(\omega_S)},
&&\mbox{\cite{KunVer00} using bubble functions (see also \S\ref{ssec_kun_})},
\\[0.3cm]
1,
&& \mbox{see  Theorem~\ref{theo_lower_struct} in \S\ref{sec_struct}}.
\end{array}
\right.
\eeq
\end{subequations}
(To be more precise, when $\varrho_{S}=1$ is used, the term $\|h_T (f-f^I)\|_{\Omega}$ in the right-hand side of \eqref{two_lower} should be replaced by a larger
$\|H_T\,{\rm osc }(f\,;T)\|_{\Omega}$; see \S\ref{sec_struct} for details.)
Importantly, the choice $\varrho_S=1$, which will be theoretically justified in \S\S\ref{sec_struct}-\ref{sec_gen}, is
consistent with the jump residual terms in our upper error estimates \eqref{first result} and \eqref{first result_improved}.

To address whether the lower error estimator $
{\mathcal E}$ in \eqref{two_lower} is sharp,
the errors $\| \nabla(u_h-u) \|_{\Omega}$ (as well as $\|h_T (f-f^I)\|_{\Omega}$) are compared with $
{\mathcal E}$
in Table~\ref{table_Kunert_lb}.
(In these computations $\nabla u$ and $f$ are  replaced, respectively, by their piecewise-linear and piecewise-quadratic interpolants.)

Clearly, the standard lower estimator with $\varrho_{S}=\frac{|S|}{{\rm diam}(\omega_S)}$ is not sharp.
Not only its effectivity indices strongly depend on the ratio $M/N$, but, perhaps more alarmingly, $
{\mathcal E}$ converges to zero as $M/N$ increases, i.e. when the mesh is anisotropically
refined in the wrong direction (while the error remains almost independent of $M/N$).
By contrast,  the estimator of \S\ref{sec_struct}, with $\varrho_{S}=1$, performs quite well, with the effectivity indices stabilizing. 

{
\begin{table}[t!]
\caption{Lower error estimators \eqref{two_lower_main} for test problem
with $u=\sin(\pi a x)$ in $\Omega=(0,1)^2$.
}\label{table_Kunert_lb}
\tabcolsep=7pt
\vspace{-0.1cm}
{\small\hfill
\begin{tabular}{ccrrrrrrr}
\hline
&&\multicolumn{3}{l}{\rule{0pt}{8pt}$\,a=1$}&&\multicolumn{3}{l}{\rule{0pt}{8pt}$\,a=3$}\\
\strut\rule{0pt}{9pt}&&
$N=20$& $N=40$&$N=80$&~&$N=20$& $N=40$&$N=80$\\
\hline
\strut&&
\multicolumn{7}{l}{\rule{0pt}{8pt}%
{Errors $\| \nabla(u_h-u) \|_{\Omega}$ 
(odd rows) ~\&~
 $\|h_T (f-f^I)\|_{\Omega}$ (even rows)}%
  }
\\
\strut\rule{0pt}{11pt}%
%
%
  $M=\;\;\;2N$ &
  &1.01e-1	&5.04e-2	&2.52e-2	&&9.00e-1	&4.52e-1	&2.27e-1\\ 	
&&3.51e-4	&4.39e-5	&5.49e-6	&
&2.83e-2	&3.55e-3	&4.45e-4\\	
 \strut\rule{0pt}{10pt}%
  $M=\;\;\;8N$ &
  &1.01e-1	&5.04e-2	&2.52e-2	&&9.00e-1	&4.52e-1	&2.27e-1\\
&&9.74e-5	&1.22e-5	&1.52e-6	&
&7.86e-3	&9.86e-4	&1.23e-4\\
\strut\rule{0pt}{10pt}%
 $M=\; 32N$ &
 &1.01e-1	&5.04e-2	&2.52e-2	&&9.00e-1	&4.52e-1	&2.27e-1\\
&&2.45e-5	&3.07e-6	&3.84e-7	&
&1.98e-3	&2.48e-4	&3.11e-5\\
\strut\rule{0pt}{10pt}%
$M=\!128N$ &
&1.01e-1	&5.04e-2	&2.52e-2	&&9.00e-1	&4.52e-1	&2.27e-1\\
&&6.14e-6	&7.67e-7	&9.59e-8	&
&4.95e-4	&6.21e-5	&7.77e-6\\
\hline
\strut&&
\multicolumn{7}{l}{\rule{0pt}{10pt}%
$
{\mathcal E}$ using $\varrho_{S}=\frac{|S|}{{\rm diam}(\omega_S)}$
(odd rows) ~\&~
  Effectivity Indices (even rows)\hspace{-0.2cm}
  }
\\
\strut\rule{0pt}{11pt}%
%
  $M=\;\;\;2N$ &
&2.80e-1	&1.40e-1	&7.02e-2	&
&2.46e+0	&1.25e+0	&6.31e-1\\		        	
        &&2.78	&2.79	&2.79	&
&2.73	&2.77	&2.78\\	
\strut\rule{0pt}{10pt}%
  $M=\;\;\;8N$ &
&1.30e-1	&6.51e-2	&3.26e-2	&	
&1.14e+0	&5.82e-1	&2.93e-1\\	
&&1.29	&1.29	&1.29	&
&1.26	&1.29	&1.29\\	
\strut\rule{0pt}{10pt}%
 $M=\; 32N$ &
&6.24e-2	&3.13e-2	&1.57e-2	&
&5.46e-1	&2.80e-1	&1.41e-1\\
&&0.62	&0.62	&0.62	&
&0.61	&0.62	&0.62\\
\strut\rule{0pt}{10pt}%
$M=\!128N$ &
&3.09e-2	&1.55e-2	&7.74e-3	&
&2.71e-1	&1.38e-1	&6.95e-2\\
&&0.31	&0.31	&0.31   &
&0.30	&0.31	&0.31\\
\hline
\strut&&
\multicolumn{7}{l}{\rule{0pt}{8pt}%
$
{\mathcal E}$ using $\varrho_{S}=1$
(odd rows) ~\&~
  Effectivity Indices (even rows)
  }
\\
\strut\rule{0pt}{11pt}%
%
%
  $M=\;\;\;2N$ &
&3.81e-1	&1.91e-1	&9.55e-2	&	
&3.34e+0	&1.71e+0	&8.58e-1\\
&&3.79	&3.79	&3.79&
&3.71	&3.77	&3.79\\	
\strut\rule{0pt}{10pt}%
  $M=\;\;\;8N$ &
&3.51e-1	&1.76e-1	&8.79e-2	&
&3.06e+0	&1.57e+0	&7.90e-1\\	
&&3.48	&3.49	&3.49&	
&3.40	&3.47	&3.49\\
\strut\rule{0pt}{10pt}%
 $M=\; 32N$ &
&3.48e-1	&1.74e-1	&8.73e-2	&	
&3.04e+0	&1.56e+0	&7.84e-1\\
&&3.46	&3.46	&3.47&	
&3.38	&3.44	&3.46\\
\strut\rule{0pt}{10pt}%
$M=\!128N$ &
&3.48e-1	&1.74e-1	&8.72e-2	&
&3.04e+0	&1.56e+0	&7.84e-1\\	        	
&&3.46	&3.46	&3.46&	
&3.38	&3.44	&3.46\\	
\hline
\end{tabular}\hfill}
\end{table}
}

{
\begin{table}[b!]\color{blue}
\caption{Lower error estimators \eqref{two_lower_main} for test problem
with $u=\sin(x/\mu)e^{-y/\eps}$ for $\mu=2\eps^2$
in  $\Omega=(0,1)\times(0,\eps)$
using $M=2N$.
}\label{table_Kunert_lb_table2}
\tabcolsep=7pt
\vspace{-0.1cm}
{\small\hfill
\begin{tabular}{lrrrrrrr}
\hline
\strut\rule{0pt}{9pt}
&
$\eps=2^{-2}$& $\eps=2^{-3}$&$\eps=2^{-4}$&~&$\eps=2^{-2}$& $\eps=2^{-3}$&$\eps=2^{-4}$\\
\hline
\strut&
\multicolumn{3}{l}{\rule{0pt}{8pt}%
Errors $\| \nabla(u_h-u) \|_{\Omega}$}&&
 \multicolumn{3}{l}{\rule{0pt}{8pt}$\|h_T (f-f^I)\|_{\Omega}$
  }
\\
\strut\rule{0pt}{11pt}%
%
%
  $N=320$
&1.66e-2	&1.60e-1	&1.74e+0	
&&2.51e-7	&2.79e-5	&2.67e-3	\\	
 \strut\rule{0pt}{10pt}%
  $N=640$
&8.30e-3	&8.01e-2	&8.73e-1	
&&3.13e-8	&3.49e-6	&3.34e-4	\\
\hline
\strut&
\multicolumn{3}{l}{\rule{0pt}{10pt}%
$\mathcal E$ using $\varrho_{S}=\frac{|S|}{{\rm diam}(\omega_S)}$}&&
\multicolumn{3}{l}{\rule{0pt}{8pt}Effectivity Indices}
\\
\strut\rule{0pt}{11pt}%
%
  $N=320$
&3.68e-2	&2.30e-1	&1.48e+0	
        &&2.22	&1.44	&0.85	\\	
\strut\rule{0pt}{10pt}%
  $N=640$
&1.84e-2	&1.15e-1	&7.47e-1			
&&2.22	&1.44	&0.86	\\	

\hline
\strut&
\multicolumn{3}{l}{\rule{0pt}{8pt}%
$\mathcal E$ using $\varrho_{S}=1$}&&
  \multicolumn{3}{l}{\rule{0pt}{8pt}Effectivity Indices}
\\
\strut\rule{0pt}{11pt}%
%
%
  $N=320$~~~
&5.76e-2	&5.55e-1	&5.92e+0	
&&3.47	&3.46	&3.40\\	
\strut\rule{0pt}{10pt}%
  $N=640$
&2.88e-2	&2.78e-1	&3.01e+0		
&&3.47	&3.47	&3.45\\

\hline
\end{tabular}\hfill}
\end{table}
}

{
\begin{table}[t!]\color{blue}
\caption{Lower error estimators \eqref{two_lower_main} for test problem
with $u=\sin((2y-x)/\eps)$
in $\Omega=(0,1)\times (0,\eps)$ using $N=M$.
}\label{table_Kunert_lb_TABLE_3}
\tabcolsep=7pt
\vspace{-0.1cm}
{\small\hfill
\begin{tabular}{lrrrrrrr}
\hline
\strut\rule{0pt}{9pt}&
$N=160$& $N=320$&$N=640$&~&$N=160$& $N=320$&$N=640$\\
\hline
\strut&
\multicolumn{3}{l}{\rule{0pt}{8pt}%
{Errors $\| \nabla(u_h-u) \|_{\Omega}$ 
}}&&
\multicolumn{3}{l}{\rule{0pt}{8pt}%
{
 $\|h_T (f-f^I)\|_{\Omega}$}%
  }
\\
\strut\rule{0pt}{11pt}%
%
%


  $\eps=2^{-4}$
  &2.29e-1	&1.14e-1	&5.72e-2	&&7.17e-5	&8.97e-6	&1.12e-6\\ 	
 \strut\rule{0pt}{10pt}%
  $\eps=2^{-5}$
  &6.67e-1	&3.34e-1	&1.67e-1	&&4.29e-4	&5.36e-5	&6.71e-6	\\
\strut\rule{0pt}{10pt}%
 $\eps=2^{-6}$
 &1.90e+0	&9.59e-1	&4.80e-1	&&2.49e-3	&3.12e-4	&3.90e-5\\

\hline
\strut&
\multicolumn{3}{l}{\rule{0pt}{10pt}%
$
{\mathcal E}$ using $\varrho_{S}=\frac{|S|}{{\rm diam}(\omega_S)}$ (odd rows)}&&
\multicolumn{3}{l}{Corresponding $\mathring{\mathcal E}$ (odd rows)}
\\
\strut&\multicolumn{3}{l}{\rule{0pt}{8pt}{
  Effectivity Indices (even rows)\hspace{-0.2cm}
  }}&&
  \multicolumn{3}{l}{$\mathring{\mathcal E}/{\mathcal E}$ (even rows)}
\\
\strut\rule{0pt}{11pt}%
%
  $\eps=2^{-4}$
&7.59e-1	&3.80e-1	&1.90e-1	&
&6.16e-2	&3.09e-2	&1.54e-2\\		        	
        &3.32	&3.32	&3.32	&
&0.08	&0.08	&0.08\\	
\strut\rule{0pt}{10pt}%
  $\eps=2^{-5}$
&2.19e+0	&1.10e+0	&5.50e-1		&	
&1.31e-1	&6.61e-2	&3.31e-2\\	
&3.28	&3.29	&3.29	&
&0.06	&0.06	&0.06\\	
\strut\rule{0pt}{10pt}%
 $\eps=2^{-6}$
&6.20e+0	&3.13e+0	&1.57e+0	&
&2.67e-1	&1.36e-1	&6.82e-2\\
&3.26	&3.27	&3.27	&
&0.04	&0.04	&0.04\\

\hline
\strut&
\multicolumn{3}{l}{\rule{0pt}{8pt}%
$
{\mathcal E}$ using $\varrho_{S}=1$ (odd rows)}&&
\multicolumn{3}{l}{Corresponding $\mathring{\mathcal E}$ (odd rows)}
\\
\strut&\multicolumn{3}{l}{\rule{0pt}{8pt}{
  Effectivity Indices (even rows)\hspace{-0.2cm}
  }}&&
  \multicolumn{3}{l}{$\mathring{\mathcal E}/{\mathcal E}$ (even rows)}
\\
\strut\rule{0pt}{11pt}%
%
%
  $\eps=2^{-4}$
&7.96e-1	&3.98e-1	&1.99e-1	&	
&2.46e-1	&1.24e-1	&6.18e-2\\
&3.48	&3.48	&3.48&
&0.31	&0.31	&0.31\\	
\strut\rule{0pt}{10pt}%
  $\eps=2^{-5}$
&2.31e+0	&1.16e+0	&5.80e-1	&
&7.43e-1	&3.74e-1	&1.87e-1	\\	
&3.46	&3.47	&3.47&	
&0.32	&0.32	&0.32\\
\strut\rule{0pt}{10pt}%
 $\eps=2^{-6}$
&6.55e+0	&3.31e+0	&1.66e+0	&	
&2.14e+0	&1.09e+0	&5.46e-1\\
&3.44	&3.46	&3.46&
&0.33	&0.33	&0.33\\
\hline

\end{tabular}\hfill}
\end{table}
}

When comparing the two estimators, note that 
their weights are similar
when $|S|\sim{\rm diam}\,\omega_S$; however,
they become dramatically different
when $|S|\ll {\rm diam}(\omega_S)$, i.e. for short edges.
Hence, our numerical experiments clearly suggest that it is the short-edge jump residual terms in the standard lower error estimator that are not sharp.


{\color{blue}
Next, consider  a two-scale exact solution
$u=\sin(x/\mu)e^{-y/\eps}$ with $\mu=2\eps^2$,
which exhibits a boundary layer in variable $y$ and smaller-scale oscillations in variable $x$.
To simplify the setting, we consider a version of problem~(\ref{eq1-1})
with this exact solution
only in the boundary-layer domain $\Omega=(0,1)\times(0,\eps)$,
with the corresponding $f$ and Dirichlet boundary conditions.
The two lower estimators from \eqref{two_lower_main} 
are compared in
Table~\ref{table_Kunert_lb_table2}
on the mesh  constructed similarly to the first test problem, with $M=2N$, only now
the 1d grid in the $y$-direction is $\{\eps\frac{j}{M}\}_{j=0}^M$.
Thus, the mesh is correctly adapted in the $y$-direction, but ignores the oscillations in the $x$-direction, i.e. it is anisotropic, but incorrectly aligned.
As $\eps$ takes smaller values, the errors increase, which is not adequately detected by the estimator with
$\varrho_{S}=\frac{|S|}{{\rm diam}(\omega_S)}$, the effectivity indices of which deteriorate (although
more moderately than in Table~\ref{table_Kunert_lb}; see Remark~\ref{rem_numerics} for further discussion).
The estimator with $\varrho_{S}=1$ again performs quite well, with all effectivity indices close to $3.45$.

Finally, in Table~\ref{table_Kunert_lb_TABLE_3}, the two estimators are tested for a one-scale exact solution
$u=\sin((2y-x)/\eps)$
on the anisotropic mesh which is incorrectly aligned in the $x$-direction only.
To simplify the setting, we again consider a version of (\ref{eq1-1})
in $\Omega=(0,1)\times (0,\eps)$ and use $N=M$.
As discussed in Remark~\ref{rem_numerics} below, both estimators exhibit stable effectivity indices.
However, the bulk contribution of the short-edge residuals, computed as
$\mathring{\EE}:=\bigl\{\sum_{S\in\mathring{\mathcal S}} \varrho_S\,|\omega_S| J_S^2\bigr\}^{1/2}$
 with $\mathring{S}:=\bigl\{|S|<\frac12{\rm diam}(\omega_S)\bigr\}$,
becomes negligible for $\varrho_{S}=\frac{|S|}{{\rm diam}(\omega_S)}$ (unlike the case $\varrho_{S}=1$).
This is undesirable, as
may
lead to the erroneous 
interpretation
that the mesh is aligned correctly and possibly requires further refinement only in the $y$-direction.
(As here we compare $\mathring{\EE}$ with the overall estimator $\EE$, it is important to note
for the component $\|h_T f^I\|_{\Omega}$ of $\EE$ that
$\|h_T f^I\|_{\Omega}/\mathring{\mathcal E}$ was $\approx 1.43$ for the first estimator and $\eps=2^{-4}$, and did not exceed $1$ in all other computations for this problem.)

\begin{remark}\label{rem_numerics}
From the  point of view of interpolation,
 if the anisotropic elements are aligned in the $x$-direction, roughly speaking,
one may expect 
that
$|J_S|$ gives an approximation to
 $O( h_y|\pt_y^2 u|+h_x|\pt^2_{xy}u|)$ for long edges
and $O( h_x|\pt_x^2 u|+h_y|\pt^2_{xy}u|)$ for short edges, where $h_x$ and $h_y$ are the mesh sizes respectively in the $x$- and  $y$-directions.
In all our computations $h_y\ll h_x$.
For the first test problem,  $\pt^2_{xy}u=\pt_y^2 u=0$,
which explains why the contributions of the short-edge residuals with correct weights are crucial for the overall efficiency of the estimator. In our second test, $h_y|\pt^2_{xy}u|$ is dominated by $h_x|\pt_x^2 u|$, but not as significantly, which is reflected in a more moderate deterioration of the estimator efficiency whenever $\varrho_S\ll1$ for short edges.
For the final test, $h_x|\pt^2_{xy}u|\simeq h_x|\pt_x^2 u|$, so $|J_S|$ takes similar in magnitude values for short and long edges; hence, even when the bulk contribution
of short-edge residuals  is almost nullified by $\varrho_S\ll1$, the overall estimator efficiency remains adequate.
\end{remark}
}

\subsection{Lower error bounds using the standard bubble approach}\label{ssec_kun_}
Here, for completeness, and with a view of 
motivating the new approach of \S\S\ref{sec_struct}-\ref{sec_gen},
we prove  a version of the lower error bound from \cite[Theorem~5.1]{KunVer00}; see also \cite[Theorem~4.37]{Ver_book_13}.
Similar bounds can also be found in \cite[Theorem~2]{Kunert2000} for the 3d case, and in
\cite[Theorem~4.3]{Kun01} for a singularly perturbed equation; see also \cite[\S4.5]{Ver_book_13}.
{\color{blue}Note also that Lemma~2.1 below gives a version of the lower error bounds from \cite[Theorem~4.3]{Kun01}, while in the earlier literature the weight $\varrho_S=\frac{|S|}{{\rm diam}(\omega_S)}$ in the bounds of type \eqref{lower_J} was replaced by the smaller $\varrho_S^2$.}

\begin{lemma}\label{lem kun_lower}
Let $\mathcal T$ satisfy the maximum angle condition, and let $|T|\sim|\omega_S|$ $\forall\, T\subset\omega_S$, $S\in {\mathcal S}\backslash\pt\Omega$.
Then for  a solution $u$ of \eqref{eq1-1} and any $u_h\in S_h$, one has%
\begin{subequations}\label{lower_f_J}
\begin{align}\label{lower_f}
h_T \|f^I\|_T&\lesssim 
\| \nabla(u_h-u) \|_{T}+ h_T\|f-f^I\|_T
&\forall T\in{\mathcal T},
\\[0.2cm]\label{lower_J}
{\textstyle\frac{|S|}{{\rm diam}(\omega_S)}}|\omega_S|J_S^2&\lesssim
\| \nabla(u_h-u) \|^2_{\omega_S}+ \|h_T (f-f^I)\|^2_{\omega_S}
&\forall S
\in{\mathcal S}\backslash\pt \Omega.
\end{align}
\end{subequations}
\end{lemma}

\begin{proof}
(i) On any $T\in\mathcal T$, consider $w:=f^I\,\phi_1\phi_2\phi_3$, where $\{\phi_i\}_{i=1}^3$  are the standard hat functions associated with the three vertices of $T$.
Now, a standard calculation yields $\|f^I\|_T^2\sim \langle f^I, w\rangle$. Note also that, in view of \eqref{eq1-1} and also
$\triangle u_h=0$ on $T$, one has
$\langle f^I, w\rangle= \langle \nabla(u-u_h), \nabla w\rangle-\langle f-f^I, w\rangle$.
Next, invoking $\|\nabla w\|_T\lesssim h_T^{-1}\|w\|_T$, one arrives at
$$
\|f^I\|_T^2\lesssim \Bigl(
 h_T^{-1}\| \nabla(u_h-u) \|_{T}+\|f-f^I\|_T \Bigr)\,
\|w\|_T
\,.
$$
The first desired result \eqref{lower_f} follows in view of $\|w\|_T\lesssim \|f^I\|_T$.

(ii) For each of the two triangles $T\subset\omega_S$, introduce a triangle $\widetilde T\subseteq T$ with an edge $S$ such that $|\widetilde T|\sim h_T|S|$.
Next, set $w:=J_S\, \color{blue}\widetilde \phi_1\widetilde \phi_2$, where $\color{blue}\widetilde \phi_1$ and $\color{blue}\widetilde \phi_2$
are the hat functions
associated with the end points of $S$
on the obtained triangulation $\{\widetilde T\}_{T\subset\omega_S}$
 (with $w:=0$ on each $T\backslash\widetilde T$ for $T\subset\omega_S$).
A standard calculation using $\triangle u_h =0$ in $T\subset\omega_S$ and \eqref{eq1-1}, yields
$$
|S|J_S^2\sim  \int_S w\,[\pt_\bnu u_h]_S =\langle \nabla u_h, \nabla w \rangle
=\langle \nabla (u_h-u), \nabla w \rangle + \langle f, w \rangle.
$$
Next, invoking $\|\nabla w\|_{T}\lesssim  h_{T}^{-1}\|w\|_{T}$ for any $T\subset\omega_S$, we arrive at
\beq\label{defic}
|S| J_S^2\lesssim\sum_{T\in\omega_S}
\underbrace{\Bigl( h_T^{-1}\| \nabla(u_h-u) \|_{T}+\|f\|_{T} \Bigr)}
_{{}\lesssim h_T^{-1}\,{\mathcal Y}^{T}_{\tiny\mbox{(\ref{lower_f})}}}
\underbrace{\|w\|_{T}}_{\sim (h_T|S|)^{1/2}|J_S|}\hspace{-0.3cm},
\hspace{1cm}
\eeq
where ${\mathcal Y}^{T}_{\scriptsize\mbox{(\ref{lower_f})}}$ denotes the right-hand side of (\ref{lower_f}), and
the latter bound was also employed
for the estimation of $\|f\|_{T}$. 
The second desired bound  \eqref{lower_J} follows
in view of $h_T=|T|/H_T\simeq|\omega_S|/{\rm diam}(\omega_S)$.
\proofend
\end{proof}

\begin{remark}\label{rem_f_I}
 The piecewise-linear Lagrange interpolant $f^I$ of $f$ used in \eqref{lower_f_J}
 may be replaced by any, possibly discontinuous, quasi-interpolant of $f$
 (such as the piecewise-constant approximation of $f$ by its element average values).
\end{remark}

\begin{remark}[Deficiency of the bubble function approach] 
\label{rem_def}
An inspection of the above proof shows that {\color{blue}it is sharp} in the sense that it cannot be tweaked to remove the weight $\frac{|S|}{{\rm diam}(\omega_S)}$ in \eqref{lower_J}
;
{\color{blue}see also Appendix~\ref{sec_app}}.
More precisely, for such an improvement, one would need $h_T\simeq|\omega_S|/|S|$ in \eqref{defic},
which is not the case for short edges.
\end{remark}

\begin{remark}[Preview of the new approach]\label{rem_new_appr}
The bubble function in the proof of \eqref{lower_J} may be viewed as a simplest local cut-off function. However, in the case of anisotropic mesh elements,
its gradient is not consistent with the diameter of the local patch.
To remedy this, when dealing with short edges in \mbox{\S\S\ref{sec_struct}-\ref{sec_gen}} below, we shall switch to a cut-off function, the support of which
comprises a larger
local patch of anisotropic elements (rather than a two-triangle patch)
and
has an interior diameter $\sim{\rm diam}(\omega_S)$.
{\color{blue}(Such local patches are highlighted in grey in Fig.\,\ref{fig_partial}\,(left) and Fig.\,\ref{fig_gen}.)}
Unsurprisingly, this approach brings new challenges. For example, we  have to deal with
multiple edges inside this larger patch;
in particular, we
need to find a way to (almost) eliminate the jump residuals associated with the long edges.
%
%
But this change of the paradigm will
lead to essentially sharper
lower error bounds of type
\eqref{two_lower} with $\varrho_S=1$.
\end{remark}

\section{Basic triangulation assumptions}\label{sec_nodes}
In the remainder of the paper,
we shall use $z=(x_z,y_z)$, $S$ and $T$ to  denote particular mesh nodes, edges and elements, respectively,
while $\mathcal N$, $\mathcal S$ and $\mathcal T$ will  denote their respective sets.
%
%
For each $z\in\mathcal N$, let
$\omega_z$ be the patch of elements {\color{blue}surrounding $z$,}
${\mathcal S}_z$ the set of edges originating at $z$,
and
$$
H_z:={\rm diam}(\omega_z),\qquad h_z:=H_z^{-1}|\omega_z|,
\qquad
\gamma_z:={\mathcal S}_z\setminus\pt\Omega.
$$
Throughout the paper we make the following triangulation assumptions.
\begin{itemize}

\item
{\it Maximum Angle condition.} Let the maximum interior angle in any triangle $T\in\mathcal T$
be uniformly bounded by some positive $\alpha_0<\pi$.
\smallskip

\item
{\it Local Element Orientation condition.}
For any $z\in\mathcal N$, there is a rectangle $\omega^*_z\supset \omega_z$ such that $|\omega^*_z|\sim |\omega_z|$.
\smallskip

\item
Also, let the number of triangles containing any node be uniformly bounded.

\end{itemize}
Note that the above 
conditions are automatically satisfied by
shape-regular triangulations.

Additionally,  we restrict our analysis to
the following two
node types
defined
using a fixed 
small constant $c_0$ (to distinguish between anisotropic and isotropic elements),
with the notation $a \ll b$ for $a<c_0b$.%
\smallskip

(1) {\it Anisotropic Nodes}, the set of which is denoted by ${\mathcal N}_{\rm ani}$, are such that
\beq\label{ani_node}
h_z\ll H_z,
\qquad\mbox{and}\qquad
|T|\sim |\omega_z| \quad\forall \ T\subset\omega_z.
\eeq
Note that the above implies that ${\mathcal S}_z$ contains at most two edges of length${}\lesssim h_z$ (see also Fig.\,\ref{fig_gen}).%
\smallskip

(2) {\it Regular Nodes}, the set of which is denoted by ${\mathcal N}_{\rm reg}$, are those surrounded by
shape-regular mesh elements.
\smallskip

The above imposes a gradual transition between anisotropic and isotropic elements, i.e.  the set ${\mathcal N}_{\rm ani}\cap{\mathcal N}_{\rm reg}$ is not necessarily empty.
(To simplify the presentation, here we exclude more general node types, such as  in \cite{Kopt15,Kopt_NM_17,Kopt17}, with both anisotropic and isotropic mesh elements allowed to appear within the same patch $\omega_z$.)

Next, recall that $\omega_S$ is the patch of two elements sharing $S$, and
introduce the set of short edges
\begin{subequations}\label{EE_YY_notation}
\begin{align}\notag
\mathring{S}&:=\{S\in{\mathcal S}\backslash\pt\Omega: |S|\ll {\rm diam}(\omega_S)\}.
\intertext{%
Now, motivated by our upper error estimate \eqref{first result_improved},
for any open domain $\mathcal D\subset\Omega$, define}
\label{YY_D}
{\mathcal Y}_{{\mathcal D}}&:=\|\nabla(u_h-u)\|_{{\mathcal D}} +\|H_T\,{\rm osc}(f\,;T)\|_{{\mathcal D}}\,,
\\[0.25cm]
\label{EE_D}
{\EE}_{\mathcal D}&:=\Bigl\{\sum_{S\in{\mathcal S}\cap{\mathcal D}}\!\!\!|\omega_S|J_S^2+\|h_T\,f^I\|_{\mathcal D}\Bigr\}^{1/2}\!\!\!\!,
\\
\label{EE0_D}
\mathring{\EE}_{\mathcal D}&:=\Bigl\{\sum_{S\in\mathring{\mathcal S}\cap{\mathcal D}}\!\!\!|\omega_S|J_S^2\Bigr\}^{1/2}\!\!\!\!.
\end{align}
\end{subequations}

\begin{remark}\label{rem_standard}
By Lemma~\ref{lem kun_lower}, one has
$
\bigl\{\sum_{S\subset{\mathcal D}\backslash\mathring{\mathcal S}}|\omega_S|J_S^2+\|h_T\,f^I\|_{\mathcal D}\bigr\}^{1/2}\lesssim {\mathcal Y}_{{\mathcal D}}$.
Indeed, this follows from \eqref{lower_f_J} combined with ${\mathcal Y}^{T}_{\scriptsize\mbox{(\ref{lower_f})}}\le {\mathcal Y}_{T}$, where
 ${\mathcal Y}^{T}_{\scriptsize\mbox{(\ref{lower_f})}}$ denotes the right-hand side of (\ref{lower_f}).
Hence, for ${\EE}_{\mathcal D}\lesssim {\mathcal Y}_{{\mathcal D}}$, it suffices to prove that
$\mathring{\EE}_{\mathcal D}\lesssim {\mathcal Y}_{{\mathcal D}}$.
\end{remark}

\section{Estimator efficiency on a partially structured anisotropic mesh}\label{sec_struct}

\subsection{Lower error bound on a partially structured anisotropic mesh}

To illustrate our approach in a simpler setting, we first present a version of the analysis for a simpler, partially structured,
anisotropic mesh in a square domain $\Omega=(0,1)^2$.
So, throughout this section,  we make the following triangulation assumptions.%

{
\begin{enumerate}
\item[A1.]
Let $\{x_i\}_{i=0}^n$ be an arbitrary mesh on the interval $(0,1)$ in the $x$ direction.
Then, let each $T\in\mathcal T$, for some $i$,\\
(i) have the shortest edge on the line segment  $\PP_i:=\{x=x_i,\,y\in[0,1]\}$;\\
(ii) have a vertex on $\PP_{i+1}$ or $\PP_{i-1}$
(see Fig.\,\ref{fig_partial}, left).
\smallskip

\item[A2.] Let ${\mathcal N}={\mathcal N}_{\rm ani}$, i.e.
each mesh node $z$ satisfies (\ref{ani_node}).
\smallskip

\item[A3.]
{\it Global Element Orientation condition.}
For any $z\in\mathcal N$, there is a rectangle $\omega^*_z\supset \omega_z$
with sides parallel to the coordinate axes
such that $|\omega^*_z|\sim |\omega_z|$.


\end{enumerate}
}

\noindent
These conditions essentially imply that all mesh elements are anisotropic and aligned in the $x$-direction.

\begin{theorem_}\label{theo_lower_struct}
Let $u$ and $u_h$ solve, respectively,  \eqref{eq1-1} and \eqref{eq1-2}
under  conditions A1--A3.
Then
in
$\Omega_i:=(x_{i-1},x_{i+1})\times(0,1)$, 
using  the notation \eqref{EE_YY_notation},
one has
\beq\label{main_struct}
\mathring{\EE}_{\Omega_i}
\lesssim
{\mathcal Y}_{\Omega_i}
\qquad \forall\, i=1,\ldots, n-1.
\eeq
\end{theorem_}

\noindent
The remainder of this section will be devoted to the proof of this result.

\begin{corollary_}\label{cor_struct_main}
Under the conditions of Theorem~\ref{theo_lower_struct}, with $\Omega_0$ and $\Omega_n$
defined using $x_{-1}:=x_0$ and $x_{n+1}:=x_n$, one has
$$
{\EE}_{\Omega_i}\lesssim {\mathcal Y}_{\Omega_i}\qquad \forall\, i=0,\ldots, n,
\qquad\qquad
{\EE}_{\Omega}\lesssim {\mathcal Y}_{\Omega}.
$$
\end{corollary_}
\begin{proof}
Combining \eqref{main_struct} with $\mathring{\EE}_{\Omega_0}=\mathring{\EE}_{\Omega_n}=0$
(as there are no short edges in $\Omega_0\cup\Omega_n$) and Remark~\ref{rem_standard}, we conclude that ${\EE}_{\Omega_i}\lesssim{\mathcal Y}_{\Omega_i}$ $\forall\,i$.
The final bound ${\EE}_{\Omega}\lesssim {\mathcal Y}_{\Omega}$ follows.
\proofend
\end{proof}

\begin{remark}[Estimator efficiency]
It follows from \cite[Theorems~5.1 and~7.4]{Kopt_NM_17} that if $f(0,y)=f(1,y)$, then $\|\nabla(u_h-u)\|_{\Omega}\lesssim {\EE}_{\Omega}+\|H_T\,{\rm osc}(f\,;T)\|_{\Omega}+\|f-f^I\|_{\Omega}$.
Comparing this upper error bound with ${\EE}_{\Omega}\lesssim {\mathcal Y}_{\Omega}$ from Corollary~\ref{cor_struct_main}, we conclude that
 the error estimator ${\EE}_{\Omega}$ is efficient up to data oscillation.
\end{remark}

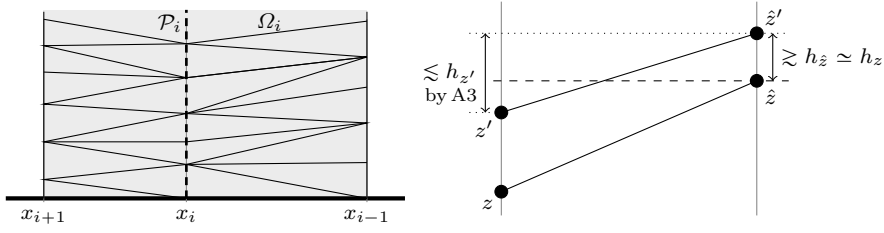
\begin{figure}[!t]
%
~\hfill
\begin{tikzpicture}[scale=0.25]

\path[fill=gray!15] (19,0)--(2,0)--(2,10)--(19,10) --cycle;
\draw[ultra thick 
] (0,0) -- (21,0) ;
\path[draw,help lines]  (19,-0.2)node[below] {$\color{black}x_{i-1}$}--(19,9.5);
\path[draw,help lines]  (9.5,-0.2)node[below] {$\color{black}x_i$}--(9.5,9.5);
\path[draw,help lines]  (2,-0.2)node[below] {$\color{black}x_{i+1}$}--(2,9.5);
\path[draw]  (19,0)--(19,9.9);
\path[draw, densely dashed,line width=1pt]  (9.5,0)--(9.5,10);
\node at (9.7,9.3)[left] {$\PP_i$};
\node at (15.0,9.35)[left] {$\Omega_i$};
\path[draw]  (2,0)--(2,9.9);
\path[draw]  (2,1)--(9.5,1.8)--(19,0)--(9.5,0)--cycle;
\path[draw]  (9.5,1.8)--(19,1.9);
\path[draw]  (2,3)--(9.5,4.5)--(19,4)--(9.5,1.8)--cycle;
\path[draw]  (2,3)--(9.5,3)--(19,4);
\path[draw]  (2,5)--(9.5,6.4)--(19,7.5)--(9.5,4.5)--cycle;
\path[draw]  (2,6.7)--(9.5,6.4);
\path[draw]  (19,5.9)--(9.5,4.5);
\path[draw]  (2,8.1)--(9.5,8.2)--(19,7.5)--(9.5,6.4)--cycle;
%
\path[draw]  (2,9.5)--(9.5,8.2)--(19,9.4);
\end{tikzpicture}
\hfill
\begin{tikzpicture}[scale=0.21]
\path[draw,help lines] (4,-1.5)--(4,12);
\path[draw,help lines] (20,-1.5)--(20,12);
\draw[fill] (4,0) circle [radius=0.4] node[below left]{$z$};
\draw[fill] (4,5) circle [radius=0.4] node[below left]{$z'$};
\draw[fill] (20,10) circle [radius=0.4] node[above right]{$\hat z'$};
\draw[fill] (20,7) circle [radius=0.4] node[below right]{$\hat z$};
\path[draw]  (4,0)--(20,7);
\path[draw] (20,10)--(4,5);
\path[draw,<->] (21,7)--(21,10) node [midway, right] {$\gtrsim h_{\hat z}\sim h_z$};
\path[draw,<->] (3,5)--(3,10)node [midway, left] {$\lesssim h_{z'}$};
\node at (3,6) [left]{{\scriptsize by\,A3}};
\path[draw,dashed] (3.5,7)--(22,7);
\path[draw,dotted] (2,10)--(22,10);
\path[draw,dotted] (2,5)--(4,5);

\node at (4,-2) {\color{white}.};
%
%
\end{tikzpicture}
\hfill~%
\caption{Partially structured anisotropic mesh (left);
illustration for Remark~\ref{rem_J_sturct} (right):
for any fixed edge $z\hat z$ and any edge $z'\hat z'$ intercepting the dashed horizontal line via $\hat z$,
the figure shows that $h_z\lesssim h_{z'}$, so there is a uniformly bounded number of edges of type $z'\hat z'$, so $\omega_z^*\subset \omega_z^{(k)}$ with $k\lesssim 1$.}
\label{fig_partial}
\end{figure}

\subsection{Preliminary results for partially structured meshes}

The following  result will be useful
in the proof of Theorem~\ref{theo_lower_struct}.

\begin{lemma}\label{lem_osc_v}
(i)
If $z\in\PP_i\backslash \pt\Omega$ for some $1\le i\le n-1$, with $\gamma_z\cap\PP_i$ formed by
the two edges $S^-$ and $S^+$, then
\beq\label{osc_v_lem}
\bigl|J_{{S}^+} - J_{{S}^-}\bigr| \lesssim h_z H_z^{-1}\!\!\!\sum_{S\in\gamma_z\backslash\PP_i}\!\!\! |J_S|.
\eeq
(ii) If $z\in\PP_i\cap \pt\Omega$ for some $1\le i\le n-1$, with $\gamma_z\cap\PP_i$ formed by
 a single edge $S^+\!$, then \eqref{osc_v_lem} holds true with
$ J_{{S}^-}\!$  replaced by~$0$.%
\end{lemma}

\begin{proof}
(i) As $z\not\in\pt\Omega$, so
$\sum_{S\in\gamma_z} \llbracket \nabla u_h\rrbracket_{S}=0$,
where $\llbracket \nabla u_h\rrbracket_{S}$ denotes
the jump in $\nabla u_h$ across any edge $S$ in $\gamma_z$ evaluated in the {anticlockwise} direction about~$z$.
Multiplying this relation by the unit vector ${\mathbf i}_{x}$ in the $x$-direction, and noting that $\llbracket \nabla u_h\rrbracket_{{S}^\pm}\cdot{\mathbf i}_{x}=\pm J_{{S}^\pm}$,
one gets the desired assertion.
Here we also use the observation that
for $S\in\gamma_z\backslash\PP_i$, one has
$|\llbracket \nabla u_h\rrbracket_S\cdot{\mathbf i}_{x}|\sim |J_S\, \bnu_S\cdot{\mathbf i}_{x}|$, where $\bnu_S$ is a unit normal vector to $S$, for which
 A3 implies
$|\bnu_S\cdot{\mathbf i}_{x}|\lesssim h_z H_z^{-1}$.

(ii) Now $z\in\pt\Omega$, so extend $u_h$ to $\R^2\backslash\Omega$ by $0$ and
imitate the above proof with the modification that now
$\sum_{S\in{\mathcal S}_z} \llbracket \nabla u_h\rrbracket_{S}=0$.
When dealing with the two edges on $\pt\Omega$,
note that for $S\in{\mathcal S}_z\cap\pt \Omega$,
one gets $\bnu_S\cdot{\mathbf i}_{x}=0$.
\proofend
\end{proof}

\begin{corollary_}\label{cor_prelim_struct}
Under the conditions of Lemma~\ref{lem_osc_v}, one has
\beq\label{prelim_struct}
|\omega_z|\,\bigl|{\textstyle\frac{H_z}{h_z}}(J_{S+}-J_{S^-})\bigr|^2\lesssim {\mathcal Y}_{\omega_z}^2,
\eeq
where ${\mathcal Y}_{\omega_z}$ is from \eqref{YY_D}, and if $z\in\PP_i\cap \pt\Omega$, then
$ J_{{S}^-}\!$  in \eqref{prelim_struct} is replaced by~$0$.
\end{corollary_}

\begin{proof}
In view of \eqref{osc_v_lem},
the left-hand side in \eqref{prelim_struct} is
$\lesssim\sum_{S\in\gamma_z\backslash\PP_i}|\omega_S|J_S^2$, where we also used $|\omega_S|\sim|\omega_z|$ $\forall\,S\in\gamma_z$.
Next, note that the set of edges $\{S\in\gamma_z\backslash\PP_i\}$ can be described as $\{S\subset\omega_z\backslash\mathring{\mathcal S}\}$,
so, by Remark~\ref{rem_standard}, the desired assertion follows.
\proofend
\end{proof}

\begin{remark}\label{rem_J_sturct}
The minimal rectangle $\omega_z^*$
from condition A3 is defined by
$\omega_z^*=(x_{i-1},x_{i+1})\times(y_z^-,y_z^+)$,
where $(y_z^-,y_z^+)$  is the
range of $y$ within
 $\omega_z$.
For this rectangle,
the above conditions (in particular A3) imply that $y^+_z-y_z^-\sim h_z$.
Furthermore,
there is $k\lesssim 1$ such that
$\omega_z^*\subset \omega_z^{(k)}$  $\forall\,z\in \mathcal{N}$,
where
$\omega_z^{(0)}:=\omega_z$, and
$\omega_z^{(j+1)}$ 
denotes the patch of elements in/touching $\omega_z^{(j)}$.
This conclusion is illustrated on Fig.\,\ref{fig_partial} (right).
(Note that $k=1$ if our partially structured  triangulation is  non-obtuse.)
\end{remark}

\subsection{Proof of Theorem~\ref{theo_lower_struct}}
\begin{proof}
Throughout the proof we shall use the somewhat simplified notation ${\mathcal Y}_{i}:={\mathcal Y}_{\Omega_i}$
and $\mathring{\mathcal E}_{i}:=\mathring{\mathcal E}_{\Omega_i}$, and also will
frequently drop the index $i$ and write $\PP:=\PP_i=\{x=x_i,\,y\in[0,1]\}$,  and $\color{blue}H:=H_i:=\frac12(x_{i+1}-x_{i-1})$.
With this notation, $\mathring{\mathcal S}\cap\Omega_i=\PP$, so, taking into consideration the structure of the mesh (see Fig.\,\ref{fig_partial}, left),
\eqref{EE0_D} and \eqref{YY_D}
with ${\mathcal D}=\Omega_i$ can be rewritten as
\beq\label{EE_J_S}
\mathring{\EE}_i^2=\sum_{S\subset\PP}\!|\omega_S|J_S^2=H\!\int_{\PP}\!\!\!J_S^2\,,
\qquad
{\mathcal Y}_{i}=\|\nabla(u_h-u)\|_{2\,;\Omega_i} +H\|{\rm osc}(f\,;T)\|_{2\,;\Omega_i}
\,.
\eeq

Next, note that for any $v\in H_0^1(\Omega)$ and $v_h\in S_h$, a standard calculation using \eqref{eq1-1},\,\eqref{eq1-2} yields
\begin{align}\notag
\underbrace{\langle\nabla (u_h-u),\nabla v\rangle
}_{{}=:\,\psi_1}&=
 \langle\nabla u_h,\nabla v\rangle-\langle  f,v\rangle\\[-0.2cm]
&{}=
 \underbrace{\langle\nabla u_h,\nabla (v-{\textstyle\frac12} v_h)\rangle
}_{{}=:\,\Psi+\frac12H^{-1}\mathring{\EE}_i^2
}
-\underbrace{\langle  f,v-{\textstyle\frac12} v_h\rangle
}_{{}=:\,\psi_2}\,.
\label{standard_calc_struct}
\end{align}
As this immediately implies $\mathring{\EE}^2_i\lesssim H(|\psi_1|+|\psi_2|+|\Psi|)$,
it suffices to prove that
\beq
\label{struct_suffice}
H\bigl(|\psi_1|+|\psi_2|+|\Psi|\bigr)
\lesssim {\mathcal Y}_i\,(\mathring{\EE}_i+{\mathcal Y}_i\bigr).
\eeq
{\color{blue}The desired assertion \eqref{main_struct} will indeed follow,
in view of ${\mathcal Y}_i\,\mathring{\EE}_i \le \theta \mathring{\EE}_i^2+ \frac14\theta^{-1}{\mathcal Y}_i^2$
with a sufficiently small positive constant $\theta$.}

The remainder of the proof is split into three parts.
In part (i), we shall describe appropriate non-standard $v_h$ and $v$, which will be crucial for \eqref{struct_suffice} to hold true.
Certain sufficient conditions for the latter will be established in part~(ii),
and then shown to be satisfied
 in part (iii).
 \smallskip

(i)
Crucially, in \eqref{standard_calc_struct}, we require that $v_h\in S_h$ and $v:=\hat v_h\not\in S_h$
 both have support in $\Omega_i$ and satisfy
\beq\label{v_h_struct}
v_h(z):={\textstyle\frac12}\!\!\!\sum_{S\in\gamma_z\cap\PP}\!\!\!\!\!J_S\;\;\forall\,z\in\PP\backslash\pt\Omega,\quad\;\;
\hat v_h(x,y):=v_h\bigl(x_i+{\color{blue}2}[x-x_i],y\bigr).
\eeq
Note that $\gamma_z\cap\PP$,
which appears in the definition of nodal values of $v_h$,
includes exactly two short edges, while, to be more precise,  $\hat v_h$ has support in
 ${\color{blue}\hat\Omega_i}:=(x_{i-1/2},x_{i+1/2})\times(0,1)\subset\Omega_i$.
 \smallskip

(ii)
We claim that for \eqref{struct_suffice}, and hence for the desired assertion \eqref{main_struct}, it suffices to prove that
the following conditions are satisfied:
\begin{subequations}\label{struct_suffice2}
\begin{align}\label{struct_suffice2_a}
\Bigl|\int_S\! (\hat v_h-{\textstyle\frac12} v_h)\Bigr|&\lesssim
\|\pt_y v_h\|_{1\,;\omega_z^*}
\quad\forall\,S\in{\color{blue}\gamma_z\backslash\PP},\;z\in\PP,
\\[0.1cm]\label{struct_suffice2_b}
\|H\nabla v_h\|_{2\,;\Omega_i}+\| v_h\|_{2\,;\Omega_i}&
\lesssim \mathring{\EE}_i+{\mathcal Y_i}\,,
\\[0.2cm]\label{struct_suffice2_c}
\sum_{S\subset\PP}|\omega_S|\Bigl\{{\textstyle\frac{H}{|S|}}{\rm osc}(v_h\,;S)\Bigr\}^2&
\lesssim {\mathcal Y}_i^2\,.
\end{align}
\end{subequations}

Indeed,
for $\psi_1$ from \eqref{standard_calc_struct},
by \eqref{EE_J_S}, 
one immediately has $|\psi_1|\lesssim {\mathcal Y}_i\|\nabla v\|_{2\,;\Omega_i}$.
Here, by \eqref{v_h_struct},  $\|\nabla v\|_{2\,;\Omega_i}=\|\nabla \hat v_h\|_{2\,;\Omega_i}\simeq\|\nabla v_h\|_{2\,;\Omega_i}$,
for which we have \eqref{struct_suffice2_b}. Combining these observations, one gets the desired bound
on $\psi_1$ in \eqref{struct_suffice}.

Next, for $\psi_2$ from \eqref{standard_calc_struct}, set $\hat f(x,y):=f(x_i+{\color{blue}2}[x-x_i],y)$ (similarly to $\hat v_h$ in \eqref{v_h_struct}).
Then 
$\frac12\langle  f, v_h\rangle=\langle  \hat f, \hat v_h\rangle=\langle  \hat f, v\rangle$,
so
\begin{align}\label{psi_2_struct}
|\psi_2|
=|\langle  f-\hat f,v\rangle|
&\le
\|f-\hat f\|_{2\,;\color{blue}\hat\Omega_i}\,\|v\|_{2\,;\hat\Omega_i}
\lesssim \|{\rm osc}(f\,;T)\|_{2\,;\Omega_i}\,\|v_h\|_{2\,;\Omega_i}\,.
\end{align}
Here we also used $\|v\|_{2\,;\Omega_i}=\|\hat v_h\|_{2\,;\Omega_i}\simeq \|v_h\|_{2\,;\Omega_i}$ (in view of \eqref{v_h_struct}),
while
 the bound on $\|f-\hat f\|_{2\,;\color{blue}\hat\Omega_i}$ follows from Remark~\ref{rem_J_sturct}.
 Combining the above with
 $H\|{\rm osc}(f\,;T)\|_{2\,;\Omega_i}\lesssim {\mathcal Y}_i$ (in view of \eqref{EE_J_S}) 
 and the bound in \eqref{struct_suffice2_b} on $\|v_h\|_{2\,;\Omega_i}$ yields the desired bound
on $\psi_2$ in \eqref{struct_suffice}.

Finally, consider $\Psi$, the most delicate term in \eqref{standard_calc_struct}.
To check that the corresponding bound in \eqref{struct_suffice} follows from \eqref{struct_suffice2},
 note that
in each triangle $T\in{\mathcal T}\cap\Omega_i$, one has $\triangle u_h=0$, so
$\int_T\nabla u_h\cdot \nabla (v-\frac12 v_h)=\int_{\pt T}\nabla u_h\cdot\bnu (v-\frac12 v_h)$.
Note also that $v=v_h=0$ on $\pt \Omega_i$,
so $\langle\nabla u_h,\nabla (v-{\textstyle\frac12} v_h)\rangle=\sum_{S\subset\Omega_i} \int_S J_S (v-{\textstyle\frac12}v_h)$.
It also follows from \eqref{v_h_struct} that $v-{\textstyle\frac12} v_h={\textstyle\frac12} v_h$ on $\PP$.
{\color{blue}Combining these observations, one gets}
\beq\label{Psi_struct_aux}
\langle\nabla u_h,\nabla (v-{\textstyle\frac12} v_h)\rangle=
{\textstyle\frac12}\int_{\PP}\!\! J_S v_h
\,+
\sum_{S\subset\Omega_i\backslash\PP}\int_S\!\! J_S(v-{\textstyle\frac12} v_h).%
\eeq
Now, subtracting $\frac12H^{-1}\mathring{\EE}_i^2=\frac12\int_{\PP}J_S^2 
$ (in view of \eqref{EE_J_S}) yields
$$
\Psi
=
{\textstyle\frac12}\int_{\PP}\!\! J_S (v_h-J_S)%
+
\sum_{S\subset\Omega_i\backslash\PP}\!\! J_S\int_S\!\! (v-{\textstyle\frac12} v_h).%
$$
So, using \eqref{EE_J_S} for the first term, and \eqref{struct_suffice2_a} combined with Remark~\ref{rem_J_sturct} for the second, one gets
\beq\label{struct_aux1}
|\Psi|\lesssim H^{-1/2} \mathring{\EE}_i \,\,\|v_h-J_S\|_{2\,;\PP}+
\Bigl\{\sum_{S\subset\Omega_i\backslash\PP}\!\!\! |\omega_S|J_S^2\Bigr\}^{1/2}\|\pt_y v_h\|_{2\,;\Omega_i}\,.
\eeq
When dealing with the second term, we also used $|\omega_z^*|\simeq|\omega_z|\simeq|\omega_S|$ for any edge $S$ originating at $z\in \PP$.
For the first term in \eqref{struct_aux1}, in view of \eqref{v_h_struct},
$\|v_h-J_S\|_{2\,;\PP}\lesssim \|{\rm osc}(v_h\,;S)\|_{2\,;\PP}\lesssim H^{-1/2}{\mathcal Y}_i$, where the latter bound follows from \eqref{struct_suffice2_c}
combined with $\frac{H}{|S|}\gtrsim 1$ and $|\omega_S|\sim H|S|$ $\forall\,S\subset\PP$.
The second term in \eqref{struct_aux1} is bounded by ${\mathcal Y}_i\cdot H^{-1}\!(\mathring{\EE}_i+{\mathcal Y_i})$,
where we used Remark~\ref{rem_standard} and \eqref{struct_suffice2_b}.
Combining these findings yields the desired bound
on $\Psi$ in \eqref{struct_suffice}.
\smallskip

(iii) To complete the proof, it remains to establish the three bounds on $v_h$ in \eqref{struct_suffice2}.
To establish \eqref{struct_suffice2_a}, for any $S\subset\gamma_z\backslash\PP$ starting at $z=(x_i,y_z)$, let $S':={\rm proj}_{y=y_z} S$,
 the projection of $S$ onto the line $y=y_z$.
Then, by \eqref{v_h_struct}, $\int_{S'}\hat v_h=\frac12\int_{S'} v_h$.
On the other hand,  by A3, one has $\Bigl|\int_S v_h-\frac{|S|}{|S'|}\int_{S'} v_h\Bigr|\lesssim \|\pt_y v_h\|_{1\,;\omega_z^*}$
and a similar bound on $\hat v_h$ (see, e.g., \cite[Lemma~7.1]{Kopt_NM_17}). Combining these  observations, and also noting that $\|\pt_y \hat v_h\|_{1\,;\omega_z^*}\simeq \|\pt_y v_h\|_{1\,;\omega_z^*}$,
yields~\eqref{struct_suffice2_a}.

For \eqref{struct_suffice2_b}, first, note that
$v_h\in S_h$ has support in $\Omega_i$, so
$\|v_h\|^2_{2\,;\Omega_i}\simeq H \| v_h\|^2_{2\,;\PP}\lesssim H \| J_S\|^2_{2\,;\PP}={\mathcal E}_i^2$,
where we used \eqref{v_h_struct} and then \eqref{EE_J_S}.
Furthermore,  $\|\nabla v_h\|_{2\,;\Omega_i}\lesssim \|\pt_y v_h\|_{2\,;\Omega_i}+H^{-1}\|v_h\|_{2\,;\Omega_i}$.
So it remains to bound  $\|\pt_y v_h\|^2_{2\,;\Omega_i}$, 
for which we note that $|\pt_y v_h|= |S|^{-1}{\rm osc}(v_h\,;S)$
on any $T$ having an edge
$S\subset\PP$ (while otherwise $\pt_y v_h=0$).
Assuming that \eqref{struct_suffice2_c} is true, one then gets
$\|H\,\pt_y v_h\|^2_{2\,;\Omega_i}\lesssim {\mathcal Y}_i^2$. Combining our findings, we conclude that \eqref{struct_suffice2_b} follows from \eqref{struct_suffice2_c}.

Finally, to establish \eqref{struct_suffice2_c},
recall \eqref{prelim_struct} and combine it with the definition of $v_h$ in \eqref{v_h_struct}
and the observation that $\sum_{z\in \PP_i}{\mathcal Y}_{\omega_z}^2\lesssim {\mathcal Y}_i^2$.
\proofend
\end{proof}

\begin{remark}[Non-smooth $f$]\label{rem_f_I_post}
An inspection of the above proof shows that \eqref{main_struct} remains valid if
in ${\mathcal Y}_{\Omega_i}$, which appears in the right-hand side,
the term
$\|H_T\,{\rm osc}(f\,;T)\|_{\Omega_i}$
is replaced by
\color{blue}$\|H_T(f-\bar f_i)\|_{\Omega_i}+\|h_T(f- f^I)\|_{\Omega_i}$,
where $\bar f_i=\bar f_i(y)$ is an arbitrary function of variable $y$.
%
To be more precise, the bound \eqref{psi_2_struct} for $\psi_2$ can be replaced by
$|\psi_2|
=|\langle  f-\bar f_i,v-\frac12 v_h\rangle|
\lesssim
\|f-\bar f_i\|_{2\,;\Omega_i}\,\| v_h\|_{2\,;\Omega_i}$.
Additionally, we use a sharper version of the bound \eqref{prelim_struct}, with
$H_T\,{\rm osc}(f\,;T)$ in the right-hand side term ${\mathcal Y}_{\omega_z}$ 
replaced by
$h_T(f- f^I)$. (This version of \eqref{prelim_struct} holds true as a similar improvement applies to the bound of Remark~\ref{rem_standard}.)
%
%
%
Note that if $f\in H^1(\Omega)$, then $\bar f_i(y)$ may be chosen equal to a 1d local average of $f$,
and if $f\in L_2(\Omega)$, then $\bar f_i(y)$ may be piecewise-constant with local 2d average values,
while
 $f^I$ may be a quasi-interpolant described in Remark~\ref{rem_f_I}.
\end{remark}


\section{Estimator efficiency on more general meshes}\label{sec_gen}

\subsection{Main result}\label{ssec_gen_mesh_main_res}

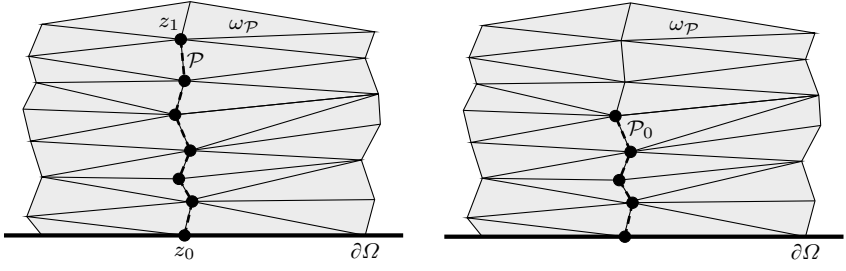
\begin{figure}[!t]

~\hfill\begin{tikzpicture}[scale=0.25]

\path[fill=gray!15](19,9.4)--(19.7,7.5)--(19.8,5.9)--(18.8,4)--(19.5,1.9)--(19,0)--%
                    (2,0)--(1.2,1)--(2,3.1)--(1.4,5)--(1,6.7)--(1.7,8.1)--(1.3,9.5)--(2,11.2)--(9.8,12.4)--(19.5,10.8)%
                    --cycle;
\draw[ultra thick 
] (0,0) -- (21,0) ;
%
\path[draw]  (9.5,0.2)--(9.9,1.8)--(9.2,3)--(9.8,4.5)--(9,6.4)--(9.5,8.2)--(9.3,10.4)--(9.8,12.4);
--(2,9.5);
\path[draw]  (19,0)--(19.5,1.9)--(18.8,4)--(19.8,5.9)--(19.7,7.5)--(19,9.4)--(19.5,10.8);
\path[draw, densely dashed,line width=1pt]  (9.5,0)--(9.9,1.8)--(9.2,3)--(9.8,4.5)--(9,6.4)--(9.5,8.2)--(9.3,10.4);

\node at (9.2,9.3)[right] {$\PP$};
\node at (8.7,11.2) {$z_1$};
\node at (13.9,11.1)[left] {$\omega_\PP$};


    \draw[fill] (9.5,0) circle [radius=0.3];
    \draw[fill] (9.9,1.8) circle [radius=0.3];
    \draw[fill] (9.2,3) circle [radius=0.3];
    \draw[fill] (9.8,4.5) circle [radius=0.3];
    \draw[fill] (9,6.4) circle [radius=0.3];
    \draw[fill] (9.5,8.2) circle [radius=0.3];
    \draw[fill] (9.3,10.4) circle [radius=0.3];

    \node at (9.5,-0.2)[below ] {$z_0$};
    \node at (19,-0.0)[below ] {$\pt \Omega$};

\path[draw]  (2,0)--(1.2,1)--(2,3.1)--(1.4,5)--(1,6.7)--(1.7,8.1)--(1.3,9.5)--(2,11.2);
\path[draw]  (1.2,1)--(9.9,1.8)--(19,0)--(9.5,0)--cycle;
\path[draw]  (9.5,1.8)--(19.5,1.9);
\path[draw]  (2,3.1)--(9.8,4.5)--(18.8,4)--(9.5,1.8)--cycle;
\path[draw]  (2,3.1)--(9.2,3)--(18.8,4);
\path[draw]  (1.4,5)--(9,6.4)--(19.7,7.5)--(9.8,4.5)--cycle;
\path[draw]  (1,6.7)--(9,6.4);
\path[draw]  (19.8,5.9)--(9.8,4.5);
\path[draw]  (1.7,8.1)--(9.5,8.2)--(19.7,7.5)--(9,6.4)--cycle;
\path[draw]  (1.3,9.5)--(9.5,8.2)--(19,9.4)--(9.3,10.4)--cycle;

\path[draw]  (2,11.2)--(9.3,10.4)--(19.5,10.8)--(9.8,12.4)--cycle;
\end{tikzpicture}%
~\hfill~
\begin{tikzpicture}[scale=0.25]

\path[fill=gray!15](19,9.4)--(19.7,7.5)--(19.8,5.9)--(18.8,4)--(19.5,1.9)--(19,0)--%
                    (2,0)--(1.2,1)--(2,3.1)--(1.4,5)--(1,6.7)--(1.7,8.1)--(1.3,9.5)--(2,11.2)--(9.8,12.4)--(19.5,10.8)%
                    --cycle;
\draw[ultra thick 
] (0,0) -- (21,0) ;
%
\path[draw]  (9.5,0.2)--(9.9,1.8)--(9.2,3)--(9.8,4.5)--(9,6.4)--(9.5,8.2)--(9.3,10.4)--(9.8,12.4);
--(2,9.5);
\path[draw]  (19,0)--(19.5,1.9)--(18.8,4)--(19.8,5.9)--(19.7,7.5)--(19,9.4)--(19.5,10.8);

\node at (13.9,11.1)[left] {$\omega_\PP$};

        \node at (9.2,6.6)[below right] {$\PP_0$};
        \path[draw, densely dashed,line width=1pt]  (9.5,0)--(9.9,1.8)--(9.2,3)--(9.8,4.5)--(9,6.4);

    \draw[fill] (9.5,0) circle [radius=0.3];
    \draw[fill] (9.9,1.8) circle [radius=0.3];
    \draw[fill] (9.2,3) circle [radius=0.3];
    \draw[fill] (9.8,4.5) circle [radius=0.3];
   \draw[fill] (9,6.4) circle [radius=0.3];

    \node at (19,-0.0)[below ] {$\pt \Omega$};

\path[draw]  (2,0)--(1.2,1)--(2,3.1)--(1.4,5)--(1,6.7)--(1.7,8.1)--(1.3,9.5)--(2,11.2);
\path[draw]  (1.2,1)--(9.9,1.8)--(19,0)--(9.5,0)--cycle;
\path[draw]  (9.5,1.8)--(19.5,1.9);
\path[draw]  (2,3.1)--(9.8,4.5)--(18.8,4)--(9.5,1.8)--cycle;
\path[draw]  (2,3.1)--(9.2,3)--(18.8,4);
\path[draw]  (1.4,5)--(9,6.4)--(19.7,7.5)--(9.8,4.5)--cycle;
\path[draw]  (1,6.7)--(9,6.4);
\path[draw]  (19.8,5.9)--(9.8,4.5);
\path[draw]  (1.7,8.1)--(9.5,8.2)--(19.7,7.5)--(9,6.4)--cycle;
\path[draw]  (1.3,9.5)--(9.5,8.2)--(19,9.4)--(9.3,10.4)--cycle;

\path[draw]  (2,11.2)--(9.3,10.4)--(19.5,10.8)--(9.8,12.4)--cycle;
\end{tikzpicture}\hfill~
\caption{A local anisotropic path $\PP$
with endpoints $z_0$ and $z_1$ (left);
$\PP_0\subset \PP$
from Theorem~\ref{theo_lower} (right).}
\label{fig_gen}
\end{figure}

{\color{blue}%
%
%
%
Under the triangulation assumptions of \S\ref{sec_nodes}, each anisotropic node
shares the local patch orientation 
with its anisotropic neighbours.
So it is not unreasonable to expect that an anisotropic mesh may include
 small clusters of anisotropic elements
sharing the same orientation.
In fact, (in particular, if
a locally anisotropic mesh was generated starting from some regular  mesh)
one may assume that the entire anisotropic part of the mesh can be split into sufficiently large, and possibly overlapping, clusters of anisotropic elements 
with interior cluster diameters $\sim{\rm diam}(\omega_z)$.

}

Hence, in this section, lower error bounds will be given for small patches of elements surrounding what will be called a local anisotropic path (also see Fig.\,\ref{fig_gen}\,(left)).%
\smallskip

{\sc Definition.} A {\it Local Anisotropic Path} $\PP$ is
a simple polygonal curve
formed by a subset of short edges, together with their endpoints,
that
does not touch any corners of $\Omega$,
has 2 endpoints (the set of the latter is denoted $\pt\PP$), 
and
satisfies the following conditions:\vspace{-0.1cm}
\begin{itemize}
\item
Any node $z\in\PP$ is anisotropic in the sense \eqref{ani_node} and satisfies
$ H_z\sim H_\PP$ for some $H_\PP$ associated with $\PP$, and also
$|\gamma_z\cap\PP|\sim h_z$ (so $\gamma_z\cap\PP$ is formed by at most two short edges).
\smallskip

\item {\it Path Element Orientation condition.}
There exists a path-specific cartesian coordinate system $(\xi,\eta)=(\xi_\PP,\eta_\PP)$ such that
for any node $z\in\PP$, there is a rectangle $\omega^*_z\supset \omega_z$
with sides parallel to the coordinate axes 
and $|\omega^*_z|\sim |\omega_z|$.
%
%
%
\end{itemize}


\begin{theorem_}[Short-edge jump residual terms]\label{theo_lower}
Suppose that
$\PP_0\subset \PP$, where
$\PP_0$ and $\PP$ are local anisotropic paths that share a coordinate system $(\xi,\eta)$,
and also
$\pt\PP\cap\pt\Omega\subset\pt\PP_0$
and
${\rm dist}(\pt\PP\backslash\pt\Omega,\,\pt\PP_0\backslash\pt\Omega)\sim H_\PP$.
%
%
Then
for   $u$ and $u_h$ satisfying respectively \eqref{eq1-1} and \eqref{eq1-2}, 
with  the notation \eqref{EE_YY_notation}, one has
\beq\label{lower_particular}
\sum_{S\subset\PP_0}\!
|\omega_S|
{J}_S^2
\lesssim
{\mathcal Y}^2_{\omega_{\PP}}\,,
\qquad
\mbox{where}\quad
\omega_\PP:=\cup_{z\in\PP}\,\omega_z\,.
\eeq
\end{theorem_}

\noindent
The remainder of this section is devoted to the proof of this result.

\begin{corollary_}\label{cor_gen_main}
Under the conditions of Theorem~\ref{theo_lower}, one has
$$
{\EE}_{\omega_{\PP_0}}\lesssim {\mathcal Y}_{\omega_{\PP}}\,.
$$
\end{corollary_}\vspace{-0.2cm}
\begin{proof}
As
\eqref{lower_particular} is equivalent to $\mathring{\EE}_{\omega_{\PP_0}}\lesssim {\mathcal Y}_{\omega_{\PP}}$,
combining the latter with Remark~\ref{rem_standard} immediately yields the desired result.
\proofend
\end{proof}

\begin{remark}[Estimator efficiency]
It follows from \cite[\S6.1 and Theorem~7.4]{Kopt_NM_17} that
under conditions on the mesh described in \S\ref{sec_nodes} and
 some additional assumptions on the orientation of anisotropic mesh elements,
 the error bound
\eqref{first result_improved} holds true, i.e.
 $\|\nabla(u_h-u)\|_{\Omega}\lesssim {\EE}_{\Omega}+\|H_T\,{\rm osc}(f\,;T)\|_{\Omega}+\|f-f^I\|_{\Omega}$.
Note that for any regular node $z\in {\mathcal N}_{\rm reg}$, \eqref{lower_f_J} yields a standard bound
${\EE}_{\omega_{z}}\lesssim {\mathcal Y}_{\omega_{z}}$.
Now, suppose that all anisotropic nodes
in ${\mathcal N}_{\rm ani}\backslash{\mathcal N}_{\rm reg}$
can be split into disjoint sets, each forming a local anisotropic path of type $\PP_0$ in Theorem~\ref{theo_lower},
{\color{blue}
and any node in $\mathcal N$ belongs to at most a finite number of the respective paths of type $\PP$.
}
Then,
in view of Corollary~\ref{cor_gen_main}, one gets
${\EE}_{\Omega}\lesssim {\mathcal Y}_{\Omega}$, i.e.
 the error estimator ${\EE}_{\Omega}$ is efficient up to data oscillation.
\end{remark}

\begin{remark}[Singular perturbation case]\color{blue}
Note that the upper a posteriori error bounds in \cite{Kopt_NM_17} 
were obtained
for more general singularly perturbed semilinear reaction-diffusion equations, solutions of which typically exhibit sharp boundary and interior layers, so anisotropic meshes are frequently employed in their numerical solution.
With regard to the lower error bounds for such equations, the standard bubble-function approach was employed in \cite{Kun01}, and, as was shown in \S\ref{sec_kunert}, the resulting estimates are not sharp even in the regular regime. Sharper lower bounds of type \eqref{lower_particular} will be generalized to this case in a forthcoming paper.%
%
\end{remark}

\subsection{Preliminary results for a local anisotropic path}

To prove Theorem~\ref{theo_lower}, we shall use a version of Lemma~\ref{lem_osc_v}, in which we shall consider
the normalized version of $J_S$ defined by
\beq\label{J_S_prime}
J_S':=J_S|\bnu_S\cdot{\mathbf i}_{\xi}|\quad\forall S\subset\PP
\qquad\Rightarrow\qquad
J_S'\sim J_S.
\eeq
Here $\PP$ is a local anisotropic path associated with the coordinate system $(\xi,\eta)$,
${\mathbf i}_{\xi}$ is the unit vector in the $\xi$-direction,
and $\bnu_S$ is a unit vector normal to $S$, while $|\bnu_S\cdot{\mathbf i}_{\xi}|\sim 1$
follows from $S$ being a short edge and the path element orientation condition.
It may be helpful to note that $J_S'$ equals a signed jump of $\pt_\xi u_h$ across $S$.

\begin{lemma}\label{lem_osc_v_gen}
Let $\PP$ be a local anisotropic path associated with the coordinate system $(\xi,\eta)$, and $J_S'$ from \eqref{J_S_prime}.
\\
(i)
For any node $z\in \PP\backslash\pt\PP$, with
 $\gamma_z\cap\PP$ formed by two edges  $S^-$ and $S^+$, 
\beq\label{osc_v_lem_prime}
\bigl|J'_{{S}^+} - J'_{{S}^-}\bigr| \lesssim h_z H_z^{-1}\!\!\sum_{S\in\gamma_z\backslash\PP}\!\!\! |J_S|.
\vspace{-0.1cm}
\eeq
(ii) 
If $z\in\pt\PP\cap\pt\Omega$, with  $\gamma_z\cap\PP$
formed by a single edge $S^+\!$, then \eqref{osc_v_lem_prime} holds true with
$ J'_{{S}^-}\!$ replaced by~$0$.%
\end{lemma}

\begin{proof}
(i) As $z\in{\mathcal N}_{\rm ani}\backslash\pt\Omega$,
 so
$\sum_{S\in\gamma_z} \llbracket \nabla u_h\rrbracket_{S}=0$,
where $\llbracket \nabla u_h\rrbracket_{S}$ denotes
the jump in $\nabla u_h$ across any edge $S$ in $\gamma_z$ evaluated in the {anticlockwise} direction about~$z$.
Multiply this relation by the unit vector ${\mathbf i}_{\xi}$ in the $\xi$-direction, and note that the quantities
$\bnu_S\cdot{\mathbf i}_{\xi}$ for $S=S^\pm$ have opposite signs (in view of the path element orientation condition combined with the maximum angle condition),
so
$|(\llbracket \nabla u_h\rrbracket_{{S}^-}+\llbracket \nabla u_h\rrbracket_{{S}^+})\cdot{\mathbf i}_{\xi}|=|J'_{{S}^+}-J'_{{S}^-}|$.
Note also that
for $S\in\gamma_z\backslash\PP$, one has
$|S|\sim H_z$ and $|\bnu_S\cdot{\mathbf i}_{\xi}|\lesssim h_z H_z^{-1}$
(again, in view of the path element orientation condition combined with the maximum angle condition), so
$|\llbracket \nabla u_h\rrbracket_S\cdot{\mathbf i}_{\xi}|= |J_S\, \bnu_S\cdot{\mathbf i}_{\xi}|\lesssim h_z H_z^{-1} |J_S|$.
Combining theses observations yields
 the desired assertion \eqref{osc_v_lem_prime}.

(ii) Now $z\in{\mathcal N}_{\rm ani}\cap\pt\Omega$, and $z$ is not a corner of $\pt\Omega$.
First, suppose that ${\mathcal S}_z\cap\pt \Omega$ is parallel to the $\xi$-axis.
Then extend $u_h$ to $\R^2\backslash\Omega$ by $0$ and
imitate the above proof with the modification that now
$\sum_{S\in{\mathcal S}_z} \llbracket \nabla u_h\rrbracket_{S}=0$.
When dealing with the two edges on $\pt\Omega$,
note that for $S\in{\mathcal S}_z\cap\pt \Omega$,
one gets $\bnu_S\cdot{\mathbf i}_{\xi}=0$.

Finally, suppose ${\mathcal S}_z\cap\pt \Omega$ is not parallel to the $\xi$-axis; then introduce a $\widetilde \xi$-axis parallel to ${\mathcal S}_z\cap\pt \Omega$.
Now the above argument yields a version of \eqref{osc_v_lem_prime} with
$J'_{{S}^+} - J'_{{S}^-}$ replaced by $\widetilde J'_S:=J_S|\bnu_S\cdot{\mathbf i}_{\widetilde\xi}|$.
The desired result follows as
$\widetilde J'_S\sim J_S\sim J_S'$.
The latter follows from $|\bnu_S\cdot{\mathbf i}_{\widetilde\xi}|\sim 1$,
in view of the path element orientation condition combined with the maximum angle condition.
\proofend
\end{proof}

\begin{corollary_}\label{cor_prelim_gen}
Under the conditions of Lemma~\ref{lem_osc_v_gen}, one has
\beq\label{prelim_gen}
|\omega_z|\,\bigl|{\textstyle\frac{H_z}{h_z}}(J'_{S+}-J'_{S^-})\bigr|^2\lesssim {\mathcal Y}_{\omega_z}^2,
\eeq
where ${\mathcal Y}_{\omega_z}$ is from \eqref{YY_D}, and if $z\in\PP\cap \pt\Omega$, then
$ J_{{S}^-}\!$  in \eqref{prelim_gen} is replaced by~$0$.
\end{corollary_}

\begin{proof} Imitate the proof of Corollary~\ref{cor_prelim_struct}.
\proofend
\end{proof}

\begin{remark}\label{rem_J_gen}
Similarly to the case of a partially structured mesh (see Remark~\ref{rem_J_sturct} and Fig.\,\ref{fig_partial} (right)),
there is $k\lesssim 1$ such that
each rectangle $\omega_z^*$ from the above path element orientation condition
satisfies
$\omega_z^*\cap\omega_\PP\subset \omega_z^{(k)}$ for all $z\in \PP$.
\end{remark}

\subsection{Proof of Theorem~\ref{theo_lower}}

We generalize the proof of Theorem~\ref{theo_lower_struct}.

\begin{proof}
Without loss of generality, let $\pt\PP=\{z_0,z_1\}$ such that $z_0\in\pt\Omega$ and $z_1\not\in\pt\Omega$ (see Fig.\,\ref{fig_gen}).
Also, to simplify the presentation, let the $\xi$-axis
be parallel to $\pt\Omega$ at $z_0$
(otherwise, see Remark~\ref{rem_xi_nonparallel}).

Set
$H:=H_\PP\sim H_{\PP_0}$.
A certain weight $\rho_S\in[0,1]$ will be associated with each $S\subset\PP$,
and
it will be imposed that $\rho_S=1$ $\forall\,S\subset\PP_0$. Hence, it suffices to prove that
\beq\label{i_claim}
\hatE_{\PP}^2:=\!\!
\sum_{S\subset\PP}\! H|S|\, \rho_S J_S J_S'
\lesssim{\mathcal Y}^2_{\omega_\PP}\,,
\eeq
where $J_S'$
is from \eqref{J_S_prime}.
Then, indeed, in view of $ H|S|\sim|\omega_S|$ and $J_S'\simeq J_S$,
\eqref{i_claim} immediately implies
the desired assertion
\eqref{lower_particular}.

Next, note that
for any $v\in H_0^1(\Omega)$ and $v_h\in S_h$, a standard calculation using \eqref{eq1-1},\,\eqref{eq1-2} yields
\begin{align}\notag
\underbrace{\langle\nabla (u_h-u),\nabla v\rangle
}_{{}=:\psi_1}&=
 \langle\nabla u_h,\nabla v\rangle-\langle  f,v\rangle\\[-0.2cm]
&{}=
 \underbrace{\langle\nabla u_h,\nabla (v-{\textstyle\frac12} v_h)\rangle
}_{{}=:\Psi+\frac12H^{-1}{\hatE}_{\PP}^2
}
-\underbrace{\langle  f,v-{\textstyle\frac12} v_h\rangle
}_{{}=:\psi_2}\,.
\label{standard_calc}
\end{align}
As this immediately implies ${\hatE}_{\PP}^2\lesssim H(|\psi_1|+|\psi_2|+|\Psi|)$,
to get
\eqref{i_claim} (and hence
the desired assertion \eqref{lower_particular}), it suffices to prove that
\beq
\label{gen_suffice}
H\bigl(|\psi_1|+|\psi_2|+|\Psi|\bigr)
\lesssim {\mathcal Y}_{\omega_{\mathcal P}}\,({\hatE}_{\PP}+{\mathcal Y}_{\omega_\PP}\bigr).
\eeq

The remainder of the proof is split into three parts.
In part (i), we shall describe appropriate
weights $\{\rho_S\}$ and
non-standard functions $v_h$ and $v$, which will be crucial for \eqref{gen_suffice} to hold true.
Certain sufficient conditions for the latter will be established in part~(ii),
and then shown to be satisfied
 in part (iii).
 \smallskip

(i)
We start by introducing a smooth monotone cut-off function $\rho$ of the arc-length parameter $l$ of $\PP$ such that
\beq\label{rho_def}
\rho=1\;\;\mbox{on~}\PP_0,\qquad
\rho=0\;\;\mbox{on~}\pt\PP\backslash\pt\Omega,\qquad
|\rho'(l)|\lesssim H^{-1}\!\sqrt{\rho(l})\;\;\forall l.
\eeq
Here for the final relation, recall that ${\rm dist}(\pt\PP\backslash\pt\Omega,\pt\PP_0\backslash\pt\Omega)\sim H$
and let $\rho$ be quadratic near its zeros.

Next, introduce
\beq\label{def_rho_s_delta_s}
\rho_S:={\textstyle \frac12}\!\!\sum_{z\in\pt S}\!\!{\color{blue}\rho(z)},
\qquad \delta_S:={\rm osc}(J_S'\,;\PP_S)
\qquad
\forall\,S\subset\PP,
\eeq
where
$J_S'=J_S|\bnu_S\cdot{\mathbf i}_{\xi}|$ is from \eqref{J_S_prime} (and also appears in \eqref{i_claim}),
and
${\PP}_S$ denotes the patch of (at most three) edges in $\PP$ touching $S$ (so $S\subset{\PP}_S\subset\PP$).

Finally,
in \eqref{standard_calc}, we let
 $v_h\in S_h$, with support in $\omega_{\PP}$, and $v:=\hat v_h\in H^1_0(\Omega)$
  satisfy
\beq\label{v_h_gen}
v_h(z):={\textstyle\frac12}\rho(z)\!\!\!\!\!\sum_{S\in\gamma_z\cap\PP}\!\!\!\!\!J'_S\;\;\forall\,z\in\PP\backslash\pt\Omega,\quad\;
\hat v_h(\xi,\eta):=v_h\bigl(
{\color{blue}\bar\xi_{\PP}(\eta)+2[\xi-\bar\xi_{\PP}(\eta)]}
,\eta\bigr).
\eeq
Here the function $\bar\xi=\bar\xi_{\PP}(\eta)\in C(\R)$ describes the curve $\PP$ for the range of $\eta$ in $\PP$, and is constant outside this range.
Without loss of generality, $\omega_{z_1}^*\subset\Omega$, so $\hat v_h$ has support in $\color{blue}\omega_{\PP}$.
(Otherwise, in view of Remark~\ref{rem_J_gen}, shorten $\PP$ by $k$ short edges starting from $z_1$, where $k\lesssim 1$.)


For $\bar\xi_{\PP}(\eta)$ in \eqref{v_h_gen}, note that $|\bar\xi'_{\PP}|\lesssim 1$ (in view of the path element orientation condition combined with the maximum angle condition).
This observation implies that $\hat v_h$ is well-defined in $\Omega$, and $\|\nabla \hat v_h\|_{2\,;{\Omega}}\sim \|\nabla v_h\|_{2\,;{\omega_\PP}}$,
as well as $\|\hat v_h\|_{2\,;{\Omega}}\sim \| v_h\|_{2\,;{\omega_\PP}}$.

Note also a few useful properties, which follow from \eqref{def_rho_s_delta_s} and \eqref{v_h_gen}:
\begin{subequations}\label{v_h_rho_delta}
\begin{align}
\label{v_h_rho_delta_b}
\Bigl|\int_S\!\! (v_h-\rho_SJ'_S)\Bigr|&\le |S|\rho_S\delta_S,
\\
\label{v_h_rho_delta_c}
\sup_{S}|v_h|&\lesssim \sup_{S\subset\PP_S}|\rho_S J'_S|,
\\[0.15cm]
{\rm osc}(v_h\,;S)
&\lesssim {\textstyle\frac{|S|}{H}}\sqrt{\rho_S}\,|J_S|+\delta_S
\qquad
\forall\,S\subset\PP.
\label{v_h_rho_delta_a}
\end{align}
\end{subequations}
To check \eqref{v_h_rho_delta_b}, note that $v_h$ is linear on $S\subset\PP$,
so $|S|^{-1}\int_S v_h$ is between $\rho_S\min_{\PP_S}\{J'_{S}\}$ and $\rho_S\max_{\PP_S}\{J'_{S}\}$,
so this assertion follows.
For \eqref{v_h_rho_delta_c}, we note that ${\textstyle\frac12}\rho(z)\le \rho_S$ for any $S\in\gamma_z\cap\PP$, so
$|v_h(z)|\le \sum_{S\in\gamma_z\cap\PP}|\rho_S J'_S|$.
Finally,
$\forall\,z\in\pt S$, where $S\subset\PP$,
one has $|v_h(z)-\rho(z)J_S'|\le \rho(z)\delta_S\le\delta_S$,
so $ {\rm osc}(v_h\,;S)\le{\rm osc}(\rho\,;S)|J'_S|+2\delta_S$.
Here $|J_S'|\le |J_S|$, while the final relationship in \eqref{rho_def} yields
${\rm osc}(\rho\,;S)\lesssim \frac{|S|}{H}\sqrt{\rho_S}$
(where we also used $\sup_S{\rho}\le 2\rho_S$ as $\rho(l)$ is monotone).
\smallskip

(ii)
We claim that for \eqref{gen_suffice}, and hence for the desired assertion \eqref{lower_particular}, it suffices to prove that
the following conditions (which give a version of \eqref{struct_suffice2}) are satisfied:
\begin{subequations}\label{gen_suffice2}
\begin{align}\label{gen_suffice2_a}
\Bigl|\int_S\! (\hat v_h-{\textstyle\frac12} v_h)\Bigr|&\lesssim
\|\nabla v_h\|_{1\,;\omega_z^*\cap\omega_{\PP}}
\quad\forall\,S\in{\color{blue}\gamma_z\backslash\PP},\;z\in\PP\,,
\\[0.1cm]\label{gen_suffice2_b}
\|H\nabla v_h\|_{2\,;{\omega_\PP}}+\| v_h\|_{2\,;{\omega_\PP}}&
\lesssim {\hatE}_{\PP}+{\mathcal Y_{\omega_\PP}}\,,
\\[0.2cm]\label{gen_suffice2_c}
\sum_{S\subset\PP}|\omega_S|\Bigl\{{\textstyle\frac{H}{|S|}}\delta_S\Bigr\}^2&
\lesssim {\mathcal Y}_{\omega_\PP}^2\,.
\end{align}
\end{subequations}

Note that  $\psi_1$ and $\psi_2$ are shown to satisfy \eqref{gen_suffice} using \eqref{gen_suffice2} in a very similar manner to the corresponding bounds in part (ii)
of the proof of Theorem~\ref{theo_lower_struct}, only for $\psi_2$ we now employ
$\hat f:=f\bigl(
{\color{blue}\bar\xi_{\PP}(\eta)+2[\xi-\bar\xi_{\PP}(\eta)]}
,\eta\bigr)$ and then
Remark~\ref{rem_J_gen}.

To show that $\Psi$ also satisfies \eqref{gen_suffice},
first, we get a version of \eqref{Psi_struct_aux} with $\Omega_i$ replaced by $\omega_{\PP}$.
Next, subtracting $\frac12H^{-1}{\hatE}_{\PP}^2=\frac12 \sum_{S\subset\PP}{\color{blue}|S|} \rho_S J_SJ'_S$
(in view of  the definition of ${\hatE}_{\PP}$ in \eqref{i_claim})
yields
$$
\Psi
=
\sum_{S\subset\PP}{\textstyle\frac12}J_S\!\!\!
\underbrace{\int_{S} (v_h-\rho_S J'_S)}_{\le |S|\rho_S\delta_S\;{\rm by\,\eqref{v_h_rho_delta_b}}}
+
\!\sum_{S\subset\omega_{\PP}\backslash\PP} \!\!\!J_S\int_S\! (v-{\textstyle\frac12} v_h).%
$$
So, using  the definition of ${\hatE}_{\PP}$ combined with $J_S\simeq J_S'$ for the first term, and  combined with Remark~\ref{rem_J_gen} for the second, one gets
\beq\label{gen_aux1}
|\Psi|\lesssim H^{-1/2} \hatE_{\PP} \,\,\|\delta_S\|_{2\,;\PP}+
\Bigl\{\!\sum_{S\subset\omega_{\PP}\backslash\PP}\!\!\!|\omega_S|J_S^2\Bigr\}^{1/2}\|\nabla v_h\|_{2\,;\omega_{\PP}}\,.
\eeq
When dealing with the second term, we also used $|\omega_z^*|\simeq|\omega_z|\simeq|\omega_S|$ for any edge $S$ originating at $z\in \PP$.
For the first term in \eqref{gen_aux1}, 
$\|\delta_S\|_{2\,;\PP}\lesssim H^{-1/2}{\mathcal Y_{\omega_\PP}}$, which follows from \eqref{gen_suffice2_c}
combined with $\frac{H}{|S|}\gtrsim 1$ and $|\omega_S|\sim H|S|$ $\forall\,S\subset\PP$.
The second term in \eqref{gen_aux1} is bounded by ${\mathcal Y_{\omega_\PP}}\cdot H^{-1}\!({\hatE}_{\PP}+{\mathcal Y_{\omega_\PP}})$,
where we used Remark~\ref{rem_standard} and \eqref{gen_suffice2_b}.
Combining these findings yields the desired bound
on $\Psi$ in \eqref{gen_suffice}.
\smallskip

(iii) To complete the proof, it remains to establish the three bounds on $v_h$ in \eqref{gen_suffice2}.
The first bound \eqref{gen_suffice2_a}
is obtained similarly to \eqref{struct_suffice2_a}. Only now
 for any $S\subset\gamma_z\backslash\PP$ starting at $z=(\xi_z,\eta_z)$, we use $S':={\rm proj}_{\eta=\eta_z} S$,
 the projection of $S$ onto the line $\eta=\eta_z$,
 and also $\|\pt_\eta \hat v_h\|_{1\,;\omega_z^*}\lesssim \|\nabla v_h\|_{1\,;\omega_z^*}= \|\nabla v_h\|_{1\,;\omega_z^*\cap\omega_{\PP}}$.

For \eqref{gen_suffice2_b}, first, note that
$v_h\in S_h$ with support in $\omega_{\PP}$, so
$\|v_h\|^2_{2\,;\omega_{\PP}}\simeq H \| v_h\|^2_{2\,;\PP}\lesssim H \|\rho_S J'_S\|^2_{2\,;\PP}\le{\hatE}_{\PP}^2$,
where we used \eqref{v_h_rho_delta_c} and also the definition of ${\hatE}_{\PP}$ in \eqref{i_claim}.
Furthermore,
on any $S\subset\PP$ one has
$|H\,\pt_l u_h|= \frac{H}{|S|}{\rm osc}(v_h\,;S)$, so
$$
\bigl\| H\nabla v_h\bigr\|^2_{2\,;\omega_{\PP}}\lesssim H\,\bigl\|\underbrace{
{\textstyle\frac{H}{|S|}}{\rm osc}( v_h\,; S)
}_{\lesssim \sqrt{\rho_S}|J_S|+\frac{H}{|S|}\delta_S \;\rm by\,\eqref{v_h_rho_delta_a}\hspace{-2cm}}
\bigr\|^2_{2\,;{\PP}}\,\,+\,\bigl\|v_h\bigr\|^2_{2\,;\omega_{\PP}}.
$$
Here $H \|\sqrt{\rho_S} J_S\|^2_{2\,;{\PP}}\lesssim {\hatE}_{\PP}$, while $H \|\frac{H}{|S|}\delta_S\|^2_{2\,;{\PP}}\lesssim  {\mathcal Y}_{\omega_\PP}^2$
assuming that \eqref{gen_suffice2_c} is true.
Combining our findings, we conclude that \eqref{gen_suffice2_b} follows from \eqref{gen_suffice2_c}.
%

Finally,  \eqref{gen_suffice2_c} is obtained similarly to \eqref{struct_suffice2_c}.  To be more precise we recall
\eqref{prelim_gen} and  combine it  with the definition of $\delta_S$ in \eqref{def_rho_s_delta_s}
and the observation that $\sum_{T\subset \omega_{\PP}}{\mathcal Y}_T^2\lesssim {\mathcal Y}_{\omega_{\PP}}^2$.
\proofend
\end{proof}

\begin{remark}\label{rem_xi_nonparallel}
If in the proof of Theorem~\ref{theo_lower}
$z_0\in\pt\Omega\cap\pt\PP$ is such that
 the $\xi$-axis is not parallel to $\pt\Omega$ at $z_0$, then one needs to tweak
 the definition of $v_h$ so that its support is in $\omega_{\PP}\backslash\omega^*_{z_0}$
 (rather than in $\omega_{\PP}$).
 This modification is required to ensure that $\hat v_h$ has support in $\color{blue}\omega_{\PP}$.
 For this, $\rho$ remains unchanged (i.e. equal to $1$) on $\PP$ near $\pt\Omega$,
 while we now set
 $v_h(z):=0$ for any $z\in\PP\cap\omega^*_{z_0}$.
 Note that the evaluations will remain without major changes as $\PP\cap\omega^*_{z_0}$ includes a {\color{blue}finite number of edges} (in view of Remark~\ref{rem_J_gen}),
 so ${\rm osc}(v_h\,;S)$ for the edge $S\subset\PP\backslash\omega^*_{z_0}$ closest to $\pt\Omega$ will involve
 ${\rm osc}(J_S'\,;\PP\cap\omega^*_{z_0})$, the estimation of which will require a finite number of applications of \eqref{prelim_gen}.
\end{remark}

\section{Conclusion}\label{sec_concl}

We have reviewed lower a posteriori error bounds obtained using the standard bubble function approach
in the context of anisotropic meshes.
Numerical examples have been given in \S\ref{sec_kunert} that clearly demonstrate that the short-edge jump residual terms in such bounds are not sharp.
Hence, in \S\S\ref{sec_struct}--\ref{sec_gen},
for linear finite element approximations of the Laplace equation in polygonal domains,
a new approach has been presented that yields essentially sharper lower a posteriori error bounds and thus shows that the upper error estimator
\eqref{first result_improved} from the recent paper \cite{Kopt_NM_17} is efficient on
{\color{blue}partially structured}
anisotropic meshes.

{\color{blue}
\appendix
\section{Generalized proof of \eqref{lower_J} for the case $\frac{|S|}{{\rm diam}(S)}\ll 1$}\label{sec_app}
The purpose of this section is to illustrate Remark~\ref{rem_def} by
giving a more general version of the proof of \eqref{lower_J} in Lemma~\ref{lem kun_lower}, which shows that
the latter proof cannot be tweaked to remove the weight $\frac{|S|}{{\rm diam}(\omega_S)}$ in \eqref{lower_J}.

\medskip

\noindent
{\it Proof of \eqref{lower_J} for the case $\frac{|S|}{{\rm diam}(S)}\ll 1$.}
As \eqref{lower_J} is obtained
in part (ii) of the proof of Lemma~\ref{lem kun_lower}, we generalize only this part.
Also, we shall consider only the case $\frac{|S|}{{\rm diam}(S)}\ll 1$.
Hence, in view of the conditions of Lemma~\ref{lem kun_lower}, one has $|S|\simeq h_T$ $\forall\,T\subset\omega_S$.

(ii)
For each of the two triangles $T\subset\omega_S$, introduce a triangle $\widetilde T\subseteq T$ with an edge $S$ such that $|\widetilde T|\sim \kappa h_T|S|$.
In the original proof, we used $\kappa=1$, while now, to allow more flexibility, it is assumed that $0<\kappa h_T\lesssim {\rm diam}(S)$.

Next, set $w:=J_S\, \widetilde \phi_1\widetilde \phi_2$, where $\widetilde \phi_2$ and $\widetilde \phi_2$
are the hat functions
associated with the end points of $S$
on the obtained triangulation $\{\widetilde T\}_{T\subset\omega_S}$
 (with $w:=0$ on each $T\backslash\widetilde T$ for $T\subset\omega_S$).
A standard calculation using $\triangle u_h =0$ in $T\subset\omega_S$ and \eqref{eq1-1}, yields
$$
|S|J_S^2\sim  \int_S w\,[\pt_\bnu u_h]_S =\langle \nabla u_h, \nabla w \rangle
=\langle \nabla (u_h-u), \nabla w \rangle + \langle f, w \rangle.
$$
Next, invoking $\|\nabla w\|_{T}\lesssim  \max\{1,\kappa^{-1}\}\,h_{T}^{-1}\,\|w\|_{T}$ for any $T\subset\omega_S$, we arrive at
$$
|S| J_S^2\lesssim\sum_{T\in\omega_S}
\Bigl( \max\{1,\kappa^{-1}\}\,h_{T}^{-1}\,\| \nabla(u_h-u) \|_{T}+
\|f\|_{\widetilde T}
 \Bigr)
\underbrace{\|w\|_{T}}_{\sim (\kappa h_T|S|)^{1/2}|J_S|}\hspace{-0.3cm}.
\hspace{1cm}
$$
%
Finally,
a calculation using $h_T\simeq|S|$ yields
$$
|S|\, |J_S|\lesssim \max\{\kappa^{1/2},\kappa^{-1/2}\}\,\,\| \nabla(u_h-u) \|_{\omega_S}+ \sum_{T\in\omega_S}(\kappa h_T|S|)^{1/2}\, \|f\|_{\widetilde T}\,.
$$
To minimize the weight $\max\{\kappa^{1/2},\kappa^{-1/2}\}$ at $\| \nabla(u_h-u) \|_{\omega_S}$ in the right-hand side, one needs $\kappa=1$, i.e. as in the original proof of \eqref{lower_J}!
Hence, we get \eqref{lower_J} with the same, i.e. unimproved, weights.
\proofend
}

\end{document}